\documentclass{article}
\usepackage{amssymb}
\usepackage{verbatim}
\usepackage{array}
\usepackage{latexsym}
\usepackage{enumerate}
\usepackage{amsmath}
\usepackage{amsfonts}
\usepackage{amsthm}
\usepackage{color}
\usepackage[english]{babel}

\newtheorem{theorem}{Theorem}[section]
\newtheorem{proposition}[theorem]{Proposition}
\newtheorem{lemma}[theorem]{Lemma}
\newtheorem{corollary}[theorem]{Corollary}

\newtheorem*{definition}{Definition}
\newtheorem{rem}[theorem]{Remark}

\def\cX{\mathcal X}

\def\min{{\rm min}}

\def\Supp{\mbox{\rm Supp}}

\newcommand{\SK}{{\tilde{\mathcal{S}}}}

\title{Weierstrass semigroups on the Skabelund maximal curve}
\date{}
\author{Peter Beelen, Leonardo Landi, and Maria Montanucci}
\begin{document}
\maketitle

\begin{abstract}
In \cite{Skabe}, D.~Skabelund constructed a maximal curve over $\mathbb{F}_{q^4}$ as a cyclic cover of the Suzuki curve. In this paper we explicitly determine the structure of the Weierstrass semigroup at any point $P$ of the Skabelund curve. We show that its Weierstrass points are precisely the $\mathbb{F}_{q^4}$-rational points. Also we show that among the Weierstrass points, two types of Weierstrass semigroup occur: one for the $\mathbb{F}_q$-rational points, one for the remaining $\mathbb{F}_{q^4}$-rational points. For each of these two types its Ap\'ery set is computed as well as a set of generators.
\end{abstract}

\noindent
AMS: 11G20, 14H05, 14H55

\vspace{1ex}

\section{Introduction}

Let $\cX$ be a nonsingular, projective algebraic curve of genus $g$ defined over a field $\mathbb{K}$. Let $P$ be a rational point of $\cX$. The Weierstrass semigroup $H(P)$ is defined as the set of natural numbers $n$ for which there exists a function $f$ on $\cX$ having pole divisor $(f)_\infty=nP$. In general, the semigroup $H(P)$ can be defined for any point $P \in \cX$ by seeing $\cX$ as an algebraic curve over the algebraic closure of $\mathbb{K}$.

According to the Weierstrass gap Theorem, see \cite[Theorem 1.6.8]{Sti}, the set $G(P) :=\mathbb{N} \setminus H(P)$ contains exactly $g$ elements called gaps. The structure of $H(P)$ in general varies as $P \in \cX$ varies. However, it is known that generically the semigroup $H(P)$ is the same, but that there can exist finitely many points of $\cX$, called Weierstrass points, with a different set of gaps. These points are of intrinsic theoretical interest, for example in St\"ohr-Voloch theory \cite{SV} to obtain characterizing properties of the curve, but when $\mathbb{K}$ is a finite field, they also occur in the study of algebraic-geometry (AG) codes \cite{TV}.

In this direction, an intensively studied class of curves is the class of maximal curves, that is, algebraic curves defined over a finite field $\mathbb{F}_q$ having as many rational points as possible according to the Hasse--Weil bound.
More precisely, an algebraic curve $\cX$ of genus $g(\cX)$ defined over $\mathbb{F}_q$ is $\mathbb{F}_q$-maximal if it has exactly $q+ 1 + 2g(\cX)\sqrt{q}$ $\mathbb{F}_q$-rational points. Clearly, this can only be the case if the cardinality $q$ of the finite field is a square or $g(\cX)=0$.

Important examples of maximal curves over suitable finite fields are the so-called Deligne-Lusztig curves; namely the $\mathbb{F}_{q^2}$-maximal Hermitian curve
\begin{equation*} \label{hermitian}
\mathcal{H}: y^{q+1}=x^q+x,
\end{equation*}
the $\mathbb{F}_{q^4}$-maximal Suzuki curve
\begin{equation} \label{suzuki}
\mathcal{S}: y^q+y=x^{q_0}(x^q+x),
\end{equation}
where $q_0=2^s$, $s \geq 1$ and $q=2q_0^2$; and the $\mathbb{F}_{q^6}$-maximal Ree curve
\begin{equation*} \label{ree}
\mathcal{R}: \begin{cases} z^q-z=x^{2q_0}(x^q-x), \\ y^q-y=x^{q_0}(x^q-x), \end{cases}
\end{equation*}
where $q_0=3^s$, $s \geq 1$ and $q=3q_0^2$.

For a fixed $g$, the curve $\mathcal{H}$ has the largest possible genus $g(\mathcal{H}) =q(q-1)/2$ that an $\mathbb{F}_{q^2}$-maximal curve can have. Also $\mathcal{H}$ and $\mathcal{S}$ are two of the four only curves of genus $g \geq 2$ having at least $8g^3$ automorphisms, see \cite[Theorem 11.127]{HKT} and \cite{Henn}.
The Weierstrass points of $\mathcal{H}$ and $\mathcal{S}$ as well as the precise structure of the Weierstrass semigroups at every point of these curves are known; see \cite{GV} and \cite{BMZ}. On the other hand nothing is known on the curve $\mathcal{R}$ and the computation of Weierstrass semigroups seems to be a challenging task, see \cite{DE}.

By a result commonly attributed to Serre, see \cite[Proposition 6]{L}, any $\mathbb{F}_{q^2}$-rational curve which is covered by an $\mathbb{F}_{q^2}$-maximal curve is also $\mathbb{F}_{q^2}$-maximal. Apart from the Deligne-Lusztig curves, most of the known maximal curves are subcovers of the Hermitian curve.

Since 2009, examples of maximal curves that are not subcovers of the Hermitian curve have been constructed as Kummer extenstions of the Deligne-Lusztig curves.
The first known example $\tilde{\mathcal{H}}$ of a maximal curve which is not a subcover of the Hermitian curve was constructed by Giulietti and Korchm\'aros as
\begin{equation*} \label{gkcurve}
\tilde{\mathcal{H}}: \begin{cases} y^{q+1}=x^q+x, \\ z^{\frac{q^3+1}{q+1}}=y^{q^2}-y; \end{cases}
\end{equation*}
see \cite{GK}. This curve is $\mathbb{F}_{q^6}$-maximal and commonly called the Giulietti-Korchm\'aros (GK) curve.
Other two families of maximal curves as generalizations of the GK curve and Kummer extensions of the Hermitian curve were constructed in \cite{BM} and \cite{GGS}.

Analogously, Skabelund \cite{Skabe} constructed Kummer extensions of the Suzuki and Ree curves as follows. Let $q_0= 2^s$ with $s\geq 1$ and $q= 2q_0^2$. The curve
\begin{equation} \label{skabesuzuki}
\tilde{\mathcal{S}}:\begin{cases} y^q+y=x^{q_0}(x^q+x),\\ t^{\frac{q^2+1}{q+2q_0+1}}=x^q+x,\end{cases}
\end{equation}
is $\mathbb{F}_{q^4}$-maximal. Now let $q_0=3^s$ with $s\geq 1$ and $q= 3q_0^2$. The curve
\begin{equation*} \label{skaberee}
\tilde{\mathcal{R}}:\begin{cases} y^q-y=x^{q_0}(x^q-x),\\ z^q-z=x^{2q_0}(x^q-x), \\ t^{\frac{q^3+1}{q+3q_0+1}}=x^q-x,\end{cases}
\end{equation*}
is $\mathbb{F}_{q^6}$-maximal.
The Weierstrass points as well as the precise structure of the Weierstrass semigroups at every point of $\tilde{\mathcal{H}}$ was determined in \cite{BM1}.
Hence it is natural to ask whether the same can be done for the curve $\tilde{\mathcal{S}}$.
In this paper, a complete answer to the aforementioned question is given.
More precisely, we show the following theorem.

\begin{theorem} \label{all}
  Let $q_0 = 2^s$ with $s \geq 1$ and $q = 2q_0^2$. Let $\tilde{\mathcal{S}}$ be the curve defined in Equation \eqref{skabesuzuki} and $P \in \tilde{\mathcal{S}}(\overline{\mathbb{F}}_{q})$. As usual denote by $H(P)$ the Weierstrass semigroup of $P$. Then the following hold:

\bigskip
\noindent
If $P \in \tilde{\mathcal{S}}(\mathbb{F}_q)$, then $H(P)=\langle q^2-2q_0 q+q, q^2-q_0 q +q_0, q^2-q+2q_0, q^2, q^2+1\rangle.$

\bigskip
\noindent
If $P \in \tilde{\mathcal{S}}(\mathbb{F}_{q^4}) \setminus \tilde{\mathcal{S}}(\mathbb{F}_q)$, then
\[H(P)=\langle q^2-q+1, q^2-2q_0+1, q^2-q_0+1, q^2, q^2+1, f_i, h_j \mid i=0,\ldots, 2q_0-2, j=0,\ldots,q_0-2 \rangle,\]
where $f_i := (i+1)q_0(q^2-q+1) - i(q^2+1) - 1$ and $h_j := (2j+1)q_0(q^2-q+1) - j(q^2+1) - q_0$.

\bigskip
\noindent
For integers $a_1,a_2,a_3,a_4,f$, we write $\sigma:=a_1+a_2+a_3+a_4+f$ and $\nu:=a_1+a_2q_0+a_32q_0+a_4q+fq^2$. If $P \not\in \tilde{\mathcal{S}}(\mathbb{F}_{q^4})$, then $H(P) =\mathbb{N} \setminus (F_1 \cup F_2 \cup F_3 \cup F_4 \cup F_5 \cup F_6)$, where
    \begin{alignat*}{3}
      & F_1 = \{ && \nu + 1 \mid 0 \leq a_1\leq q_0-1; 0 \le a_2 \le 1; 0\le a_3 \leq q_0-1; a_4 \ge 0; f \geq 0;\sigma \leq q-2 \}, \\
      & F_2 = \{ && \nu + (n+1)q_0 q  + 1 \mid 1 \leq n \leq 2q_0-2; 0 \leq a_1\leq q_0-1; 0 \le a_2 \le 1;0 \leq  a_3 \leq q_0-1; \\
      & && 0 \leq a_4 \leq q-q_0-1-nq_0; f \geq 0; \sigma = q-q_0-2-nq_0+n \}, \\
      & F_3 = \{ && \nu+(2n+1)q_0q+n+2 \mid 0 \leq n \leq q_0-2; 0 \leq a_1 \leq q_0-2-n;0 \le a_2 \le 1; 0 \leq a_3\leq q_0-1;\\
      & &&  0 \le a_4 \leq q_0-1; f \geq 0; \sigma = q-q_0-2-2nq_0+n \}, \\
      & F_4 = \{ && \nu + (2n+2)q_0q+n+3 \mid 0 \leq n \leq q_0-3; 0 \leq a_1 \leq q_0-3-n; 0 \le a_2 \le 1;0 \leq a_3\leq q_0-1; \\
      & &&  0\le a_4 \leq q_0-1; f \geq 0; \sigma = q-2q_0-2-2nq_0+n \}, \\
      & F_5 = \{ && \nu + c q_0(q+1) + d(2qq_0+2q_0+1)  + 1 \mid a_1=0; 0 \le c \le 1; 0 \leq a_2 \leq 1-c;1-c \leq d \leq q_0-1; \\
      & &&  0 \leq a_3 \leq q_0-1-d; 0 \leq a_4 \leq q_0-1; f \geq 0; \sigma = q-2-2dq_0-cq_0 \}, \\
      & F_6 = \{ && \nu + q_0 + (2n+2)q_0q+n+2 \mid a_1=0; a_2=0; 0 \leq n \leq q_0-2;0 \leq a_3 \leq n; 0 \leq a_4 \leq q_0-1;  \\
      & &&  f \geq 0; \sigma = q-2q_0-2-2nq_0+n \}.
    \end{alignat*}
\end{theorem}
As a result, we will also obtain the set of Weierstrass points of $\tilde{\mathcal{S}}$.
\begin{corollary}
The set of Weierstrass points of $\tilde{\mathcal{S}}$ is equal to $\tilde{\mathcal{S}}(\mathbb{F}_{q^4})$.
\end{corollary}

The paper is organized as follows: In the next section we give the necessary background on the curve $\tilde{\mathcal{S}}$ as well as some results on Weierstrass semigroups and their gaps that we will need later. In Section \ref{sec:three}, we compute $H(P)$ for $P \in \tilde{\mathcal{S}}(\mathbb{F}_q)$. Here it should be mentioned that to a large extent this case was already treated in \cite{Skabe}. Our results complement those in \cite{Skabe} by proving a claim on the generators of the semigroup that was stated in \cite{Skabe} without proof. For $P \not\in \tilde{\mathcal{S}}(\mathbb{F}_q)$, the corresponding Weierstrass semigroups are currently unknown. Our main results are the determination of these semigroups. More precisely, in Section \ref{sec:four} we compute the Weierstrass semigroup for $P \in \tilde{\mathcal{S}}(\mathbb{F}_{q^4}) \setminus \tilde{\mathcal{S}}(\mathbb{F}_q)$, while in Section \ref{sec:five} we deal with the generic case $P \not\in \tilde{\mathcal{S}}(\mathbb{F}_{q^4})$.

\section{The curve $\SK$}\label{sec:two}

Let $q_0=2^s$ with $s \in \mathbb{N}$ and let $q=2q_0^2$. The curve $\SK$ defined in Equation \eqref{skabesuzuki} was constructed by Skabelund in \cite{Skabe}. 
%
This curve has genus $g(\SK)=\frac{q^3-2q^2+q}{2}$, $q^5-q^4+q^3+1$ $\mathbb{F}_{q^4}$-rational points, and a unique point at infinity $P_{\infty}$, which is singular and $\mathbb{F}_{q}$-rational. The curve $\SK$ has been introduced in \cite{Skabe}, where it was proved that $\SK$ is maximal over $\mathbb{F}_{q^4}$.
As is clear from Equations \eqref{suzuki} and \eqref{skabesuzuki}, the curve $\SK$ is a cyclic Galois cover of the Suzuki curve of degree $q-2q_0+1$ by projecting $(x,y,t)$ on $(x,y).$ In the remainder of this paper, we will denote this cover by $\mathrm{pr}: \SK \to \mathcal S.$
The cover $\mathrm{pr}$ gives in the usual way rise to a map $\mathrm{pr}^*: Div(\mathcal S) \to Div(\SK).$ If $Q \in \mathcal{S},$ then for any $Q$ of $\mathcal S$, the divisor $\mathrm{pr}^*(Q)$ is equal to the orbit of any point $P \in \mathrm{pr}^{-1}(Q)$ under the cyclic group $C_{q-2q_0+1}$ multiplied with the order of the stabilizer of $P$ in $C_{q-2q_0+1}$. For an algebraic function $f$ on $\mathcal S$ (resp. $\SK$), we will write $(f)_{\mathcal S}$ (resp. $(f)_{\SK}$) for its divisor. It is well known that for any function $f \in \mathbb{F}_q(\mathcal S)$, it holds that  $(f)_\SK=\mathrm{pr}^* ((f)_{\mathcal S}).$

The automorphism group ${\rm Aut}(\SK)$ of $\SK$ is defined over $\mathbb{F}_{q^4}$ and has size $q^2(q-1)(q^2+1)(q-2q_0+1)$.  It has a normal subgroup $H$ isomorphic to the Suzuki group ${\rm Sz}(q)$ and ${\rm Aut}(\SK)=H\times C_{q-2q_0+1}$, where $C_{q-2q_0+1}$ is the Galois group of the cyclic Galois-covering $\mathrm{pr}: \SK \rightarrow \mathcal{S}$.
The set $\SK(\mathbb{F}_{q^4})$ of the $\mathbb{F}_{q^4}$-rational points of $\SK$ splits into two orbits under the action of ${\rm Aut}(\SK)$:
one orbit has size $q^2+1$ and equals $\mathcal O_1=\SK(\mathbb{F}_q)$. The other orbit is $\mathcal O_2=\SK(\mathbb{F}_{q^4})\setminus\SK(\mathbb{F}_q)$ and has size $q^5-q^4+q^3-q^2$; see \cite{GMQZ} for these and other details on $\SK$. In particular, we can conclude that any point $P \in \mathcal O_1=\SK(\mathbb{F}_q)$ has the same Weierstrass semigroup and similarly for $P \in \mathcal O_2=\SK(\mathbb{F}_{q^4})\setminus\SK(\mathbb{F}_q).$

Let $x,y,z,t\in \mathbb{F}_{q^4}(\SK)$ be the coordinate functions of the function field of $\SK$, which satisfy $y^q+y=x^{q_0}(x^q+x)$ and $t^{q-2q_0+1}=x^q+x$. These functions have a pole at $P_\infty$ only. The same is true for the functions $z:=y^{2q_0}+x^{2q_0+1}$ and $w:=xy^{2q_0}+z^{2q_0}.$
%
%
Now let $P \neq P_\infty$ be the point of $\SK$ with $(x,y,t)$-coordinates $(a,b,c)$. Further let $Q:=\mathrm{pr}(P)$. Then $Q$ has $(x,y)$-coordinates $(a,b).$ With $\Phi(P)$ (resp. $\Phi(Q)$) we denote the point $\Phi(P)=(a^q,b^q,c^q)$ (resp. $\Phi(Q)=(a^q,b^q)$).
In the sequel we will frequently use the following functions and information on their divisors:
\begin{equation}\label{xtilda}
\tilde{x}_Q:=x+a; \quad (\tilde{x}_Q)_{\mathcal S}=Q+E_x-q Q_\infty,
\end{equation}
where $E_x \ge 0 \text{ and } \Supp(E_x) \cap \{Q,Q_\infty\}=\emptyset$,

\begin{equation} \label{ytilda}
\tilde{y}_Q:=y+b+a^{q_0}(x+a); \quad (\tilde{y}_Q)_{\mathcal S}=q_0Q+\Phi(Q)+E_y-(q+q_0) Q_\infty,
\end{equation}
where $E_y \ge 0 \text{ and } \Supp(E_y) \cap \{Q,\Phi(Q),Q_\infty\}=\emptyset$,
\begin{equation} \label{ztilda}
\tilde{z}_Q:=a^{2q_0}x+z+b^{2q_0}; \quad (\tilde{z}_Q)_{\mathcal S}=2q_0Q+\Phi(Q)+E_z-(q+2q_0) Q_\infty,
\end{equation}
where $E_z \ge 0 \text{ and } \Supp(E_z) \cap \{Q,\Phi(Q),Q_\infty\}=\emptyset$,
and
\begin{equation} \label{wtilda}
\tilde{w}_Q:=a^q \tilde{z}_Q+b^{2q_0}x+w+b^2+a^{2q_0+2};  \quad (\tilde{w}_Q)_{\mathcal S}=qQ+2q_0\Phi(Q)+\Phi^2(Q)-(q+2q_0+1) Q_\infty.
\end{equation}
Over $\mathbb{F}_q$, the Suzuki curve has $L$-polynomial $(1+2q_0T+qT^2)^{g(\mathcal S)}.$ By the Fundamental Equation \cite[Proposition 10.6 (I)]{HKT}, the operator $q\,\mathrm{id}+2q_0\Phi+\Phi^2$ therefore acts trivially on linear equivalence classes of divisors of degree zero, for example on $Q-Q_\infty$. This means that the existence of a function $\tilde{w}_Q$ with divisor $(q\,\mathrm{id}+2q_0\Phi+\Phi^2)(Q-Q_\infty)$ is guaranteed.
Similarly applying the Fundamental Equation to the maximal curve $\SK$ (over $\mathbb{F}_{q^4}$), we see that there exists a function $\pi_P$ on $\SK$ with divisor:
\begin{equation}
  \label{eq:div_pi}
  (\pi_P)_\SK = q^2P+\Phi^4(P)-(q^2+1)P_\infty.
\end{equation}
In particular $(\pi_P)_\SK = (q^2+1)(P-P_\infty)$ if $P$ is $\mathbb{F}_{q^4}$-rational.
It is not hard to find the function $\pi_P$ explicitly:
$$\pi_P=\tilde{w}_Q+A^q\tilde{z}_Q+A^{q+2q_0}\tilde{x}_Q+c^{q^2}(t+c), \text{ where } A:=a^q+a.$$
Note that if $P$ is $\mathbb{F}_q$-rational, then $A=0$ and $\pi_P=\tilde{w}_Q.$ This observation is consistent with the divisors given above. Indeed, if $\Phi(P)=P$, then also $\Phi(Q)=Q$ and since $Q$ and $Q_\infty$ are totally ramified in the cover $\mathrm{pr}: \SK \to \mathcal{S}$:
\[(\pi_P)_\SK=\mathrm{pr}^*((\tilde{w}_Q)_{\mathcal S})=(q+2q_0+1)(q-2q_0+1)(P-P_\infty)=(q^2+1)(P-P_\infty).\]

As we already announced, our aim is to determine the structure of the Weierstrass semigroup at every point $P$ of the curve $\tilde{\mathcal{S}}$.
First of all we observe that the orbit structure of ${\rm Aut}(\SK)$ tells us already that the Weierstrass semigroup at $P \in \SK(\mathbb{F}_q)$ will be the same as $H(P_\infty)$, and  $H(R)=H(Q)$ for all $R,Q \in \SK(\mathbb{F}_{q^4})\setminus\SK(\mathbb{F}_q)$.

Clearly, as we will do for $P \not\in \tilde{\mathcal{S}}(\mathbb{F}_{q^4})$, computing $H(P)$ is equivalent to completely determine the gap structure at $P$, that is, the complement $G(P)=\mathbb{N} \setminus H(P)$. Hence to finish this section we state some facts that we will use to achieve this. We start with the following well-known proposition connecting regular differentials (i.e., differential forms having no poles anywhere on $\tilde{\mathcal{S}}$) and gaps of $H(P)$.

\begin{proposition} \cite[Corollary 14.2.5]{VS} \label{differentials}
Let $\cX$ be an algebraic curve of genus $g$ defined over a field $\mathbb{K}$. Let $P$ be a point of $\cX$ and $\omega$ be a regular differential on $\cX$. Then $v_P(\omega) + 1$ is a gap at $P$.
\end{proposition}

This proposition has the following, for us very useful, consequence.

\begin{corollary} \label{basisdiff}
For any point $P$ on the curve $\tilde{\mathcal{S}}$ distinct from $P_\infty$, and for any $f\in L((2g(\tilde{\mathcal{S}})-2)P_\infty)$,we have $v_P(f) + 1 \in G(P)$.
\end{corollary}

\begin{proof}
First note that $(dx)_{\mathcal{S}}= (2g(\mathcal{S})-2)Q_\infty=(2q_0(q-1)-2)Q_\infty$. The set of points that branch in the cover $\mathrm{pr}: \SK \to \mathcal{S}$ is exactly $\mathcal{S}(\mathbb{F}_{q})$, the set of $\mathbb{F}_q$-rational points of the Suzuki curve, all with ramification index $q-2q_0+1$. Moreover, the points of $\SK$ above $\mathcal{S}(\mathbb{F}_{q})$ are precisely the $\mathbb{F}_q$-rational points of $\tilde{\mathcal{S}}$. Therefore, we immediately obtain that
\begin{eqnarray*}
(dx)_{\tilde{\mathcal{S}}}&=& (q-2q_0+1)(2q_0(q-1)-2)P_\infty+ (q-2q_0)\sum_{P \in \SK(\mathbb{F}_{q})} P
\\ &=& (2q_0q^2-2q^2+q-2)P_\infty+ (q-2q_0)\sum_{P \in \SK(\mathbb{F}_{q}), \ P \ne P_\infty } P.
\end{eqnarray*}
Thus, from $t^{q-2q_0+1}=x^q+x$ and $(t)_{\tilde{\mathcal{S}}}=\sum_{P \in \SK(\mathbb{F}_{q}), \ P \ne P_\infty } P - q^2 P_\infty$ we get,
$$(dt)_{\tilde{\mathcal{S}}}= \bigg(\frac{dx}{t^{q-2q_0}}\bigg)_{\tilde{\mathcal{S}}}= (q^3-2q^2+q-2)P_\infty.$$
In particular a differential $fdt$ is regular if and only if $f \in L((q^3-2q^2+q-2)P_\infty) =L((2g(\tilde{\mathcal{S}})-2)P_\infty)$. The claim now follows by applying Proposition \ref{differentials}.
\end{proof}

\section{The Weierstrass semigroup for $P\in \SK(\mathbb{F}_q)$}\label{sec:three}

From Equations \eqref{xtilda}-\eqref{wtilda} and $(t)_{\tilde{\mathcal{S}}}=\sum_{P \in \SK(\mathbb{F}_{q}), \ P \ne P_\infty } P - q^2 P_\infty$, it is clear that the Weierstrass semigroup $H(P_\infty)$ contains the integers $q^2-2q_0 q+q$, $q^2-q_0 q +q_0$, $q^2-q+2q_0$, $q^2$, and $q^2+1.$
This was already observed in \cite{Skabe}, where in fact it was written that $H(P_\infty)$ is generated by these five numbers.
To complete the results in \cite{Skabe}, we give a formal proof of this claim in this section.
Note that the fact that $\SK(\mathbb{F}_q)$ forms one orbit under the action of the automorphism group of $\SK$, directly implies that $H(P)=H(P_\infty)$ for any point $P \in \SK(\mathbb{F}_q).$

We will use the following standard terminology for numerical semigroups:
\begin{definition}
Let $Z \subset \mathbb{N}$ be a numerical semigroup.
The set $G(Z):=\mathbb{N} \setminus Z$ is called the set of gaps of $Z$.
The \emph{genus} $g(Z)$, resp. \emph{multiplicity} $m_Z$, resp. \emph{conductor} $c_Z$ of $Z$ is defined to be
\[g(Z):=|\mathbb{N} \setminus Z|, \quad \text{resp.} \quad m_Z:=\min\{z \in Z \mid z>0\}, \quad \text{resp.} \quad c_Z:=1+\max\{z \in \mathbb{N} \setminus Z\}.\]
Further, the Ap\'ery set $Ap(Z)$ of $Z$ is defined to be
\[\mathrm{Ap}(Z) := \{ z \in Z \mid z-m_Z \not\in Z \}.\]
\end{definition}

The Ap\'ery set $\mathrm{Ap}(Z)$ has cardinality $m_Z$ and its elements form a complete set of representatives for the congruence classes of $\mathbb{Z}$ modulo $m_Z$. Moreover, by definition of $\mathrm{Ap}(Z)$, each representative is minimal among those lying in $Z$. As a consequence, the semigroup $Z$ can conveniently be described as $Z = \{ a + t m_Z \mid a \in \mathrm{Ap}(Z), t \geq 0 \}.$
This also implies that the genus and conductor of $Z$ can be deduced directly from $m_Z$ and $\mathrm{Ap}(Z)$:
\[g(Z) = \sum_{a \in \mathrm{Ap}(Z)} \left\lfloor \frac{a}{m_Z} \right\rfloor \quad \text{and} \quad c_Z=1+\max\{z \in \mathrm{Ap}(Z)\}-m_Z.\]
This implies that if $A \subset Z$ consists of $m_Z$ elements that form a complete set of representatives for the congruence classes of $\mathbb{Z}$ modulo $m_Z$, then $A=\mathrm{Ap}(Z)$ if and only if $g(Z) = \sum_{a \in A} \left\lfloor \frac{a}{m_Z} \right\rfloor.$
It is well known that $c_Z \le 2g(Z)$. A numerical semigroup for which $c_Z=2g(Z)$ is called a \emph{symmetric} semigroup.

Now we return to the study of the semigroup $H(P_\infty).$
It will be convenient to introduce a notation for the five integers contained in it that we mentioned previously:
\[\mathfrak{g}_0:=q^2-2q_0 q+q, \quad \mathfrak{g}_1:=q^2-q_0 q +q_0, \quad \mathfrak{g}_2:=q^2-q+2q_0, \quad \mathfrak{g}_3:=q^2, \quad \text{and} \quad \mathfrak{g}_4:=q^2+1.\]
Moreover, we write $H=\langle \mathfrak{g}_0,\mathfrak{g}_1,\mathfrak{g}_2,\mathfrak{g}_3,\mathfrak{g}_4\rangle$ for the numerical semigroup generated by them.
Since $\mathfrak{g}_0 < \mathfrak{g}_1 < \mathfrak{g}_2 < \mathfrak{g}_3 < \mathfrak{g}_4$, we see that $m_H=\mathfrak{g}_0$.
Moreover, since $H \subseteq H(P_\infty)$, we have $g(H) \ge g(\SK)$.
Using Ap\'ery sets, we will actually show that $g(H) \le g(\SK)$, which then immediate will imply that $H=H(P_\infty).$

\begin{lemma}
  \label{lem:apery}
Let $A := \{ h \mathfrak{g}_1 + i \mathfrak{g}_2 + j \mathfrak{g}_3 + k \mathfrak{g}_4 \mid 0 \leq h \leq 1, 0 \leq i \leq q_0-1, 0 \leq j \leq q-2q_0, 0 \leq k \leq q_0-1 \}$.
Then $A$ is a complete set of representatives for the congruence classes of $\mathbb{Z}$ modulo $\mathfrak{g}_0$.
\end{lemma}

\begin{proof}
  By definition of $A$, there are $2q_0^2(q-2q_0+1) = \mathfrak{g}_0$ different four-tuples of parameters $(h,i,j,k)$ describing the elements of $A$, so $|A| \leq \mathfrak{g}_0$. The lemma follows once we show that elements of the form $h \mathfrak{g}_1 + i \mathfrak{g}_2 + j \mathfrak{g}_3 + k \mathfrak{g}_4$ are pairwise distinct modulo $\mathfrak{g}_0$ for distinct four-tuples $(h,i,j,k)$ satisfying
  $$0 \leq h \leq 1, \quad 0 \leq i \leq q_0-1, \quad 0 \leq j \leq q-2q_0, \quad 0 \leq k \leq q_0-1.$$
 Now let $a = h \mathfrak{g}_1 + i \mathfrak{g}_2 + j \mathfrak{g}_3 + k \mathfrak{g}_4$ and $a^\prime = h^\prime \mathfrak{g}_1 + i^\prime \mathfrak{g}_2 + j^\prime \mathfrak{g}_3 + k^\prime \mathfrak{g}_4$ be elements of $A$ and suppose that $a \equiv a^\prime \pmod{\mathfrak{g}_0}.$
 Recall that $q=2q_0^2$ and that $q$ divides $\mathfrak{g}_0.$ Hence the congruence $a \equiv a^\prime \pmod{\mathfrak{g}_0}$ implies that $a \equiv a^\prime \pmod{q_0}.$ This in turn is equivalent to the congruence $k \equiv k^\prime \pmod{q_0}$. Since $0 \le k < q_0$ and $0 \le k^\prime < q_0,$ we see that $k=k^\prime.$ Now working modulo $2q_0$, implies that $q_0h \equiv q_0 h^\prime \pmod{2q_0}.$ Since $0 \le h \le 1$ and $0 \le h^\prime \le 1,$ we obtain $h=h^\prime.$ Similarly working modulo $q$, implies that $i=i^\prime.$ At this point, we know that $j \mathfrak{g}_3 \equiv j^\prime \mathfrak{g}_3 \pmod{\mathfrak{g}_0}$. Dividing by $q$, we obtain that $j q \equiv j^\prime q \pmod{q-2q_0+1},$ whence $j \equiv j^\prime \pmod{q-2q_0+1}$ and therefore $j=j^\prime.$
\end{proof}

\begin{theorem}\label{main1}
For any $P \in \SK(\mathbb{F}_q)$, we have $H(P)=\langle q^2-2q_0 q+q, q^2-q_0 q +q_0, q^2-q+2q_0, q^2, q^2+1 \rangle.$
\end{theorem}
\begin{proof}
As before, let $H$ denote the semigroup $\langle q^2-2q_0 q+q, q^2-q_0 q +q_0, q^2-q+2q_0, q^2, q^2+1 \rangle.$
As observed before, we know that $H(P)$ contains $H$ and hence $g(H) \ge g(\SK).$
From Lemma \ref{lem:apery}, we see that $g(H) \le \sum_{a \in A} \lfloor \frac{a}{\mathfrak{g}_0}\rfloor.$ Note that equality holds if and only if $A=\mathrm{Ap}(H).$

For simplicity let us denote $m = 2q_0^2-2q_0+1$, so that $\mathfrak{g}_0 = 2q_0^2 m$. Then:
  \begin{alignat*}{2}
    & \sum_{a \in A} \left\lfloor \frac{a}{\mathfrak{g}_0} \right\rfloor = \sum_{a \in A} \frac{a - a \text{ mod } \mathfrak{g}_0}{\mathfrak{g}_0} = \frac{1}{\mathfrak{g}_0} \sum_{a \in A} a - \frac{1}{\mathfrak{g}_0} \sum_{r=0}^{\mathfrak{g}_0-1} r \\
    & = \frac{1}{\mathfrak{g}_0} \sum_{h=0}^1 \sum_{i=0}^{q_0-1} \sum_{j=0}^{m-1} \sum_{k=0}^{q_0-1} ( h \mathfrak{g}_1 + i \mathfrak{g}_2 + j \mathfrak{g}_3 + k \mathfrak{g}_4 ) - \frac{1}{\mathfrak{g}_0} \cdot \frac{\mathfrak{g}_0(\mathfrak{g}_0-1)}{2} \\
    & = \frac{1}{\mathfrak{g}_0}\left(q_0^2m\mathfrak{g}_1+2q_0m \frac{q_0(q_0-1)}{2} \mathfrak{g}_2+2q_0^2 \frac{m(m-1)}{2} \mathfrak{g}_3+2q_0 m \frac{q_0(q_0-1)}{2} \mathfrak{g}_4 \right)-\frac{\mathfrak{g}_0-1}{2}\\
 &=g(\SK).
  \end{alignat*}
Hence $g(H) \le g(\SK)$. Combining all the above, we see $g(H)=g(\SK),$ which implies that $H(P)=H.$
\end{proof}

\begin{corollary}
Let $P \in \SK(\mathbb{F}_q)$ and as before, let $A := \{ h \mathfrak{g}_1 + i \mathfrak{g}_2 + j \mathfrak{g}_3 + k \mathfrak{g}_4 \mid 0 \leq h \leq 1, 0 \leq i \leq q_0-1, 0 \leq j \leq q-2q_0, 0 \leq k \leq q_0-1 \}$. Then $\mathrm{Ap}(H(P))=A$. Further, $H(P)$ is a symmetric semigroup.
\end{corollary}
\begin{proof}
The statement $\mathrm{Ap}(H(P))=A$ follows from the proof of Theorem \ref{main1}. In particular, the largest gap of $H(P_\infty)$ is
  $$ \max\{z \in A\} - \mathfrak{g}_0 = \mathfrak{g}_1 + (q_0-1)\mathfrak{g}_2 + (q-2q_0)\mathfrak{g}_3 + (q_0-1)\mathfrak{g}_4 - \mathfrak{g}_0 =2g(\SK)-1. $$
This means that the conductor of $H(P)$ is $2g(\SK)$ and hence that the semigroup is symmetric.
\end{proof}

We will later see that any $P \in \SK(\mathbb{F}_q)$ is a Weierstrass point by computing the Weierstrass semigroups for all points on $\SK$. However, using the symmetry of $H(P_\infty)$ that we just proved and \cite[Proposition 50]{Konto}, we can already conclude this now.
%
%
%
%

\section{The Weierstrass semigroup for $P \in \SK(\mathbb{F}_{q^4}) \setminus \SK(\mathbb{F}_q)$}\label{sec:four}

Let $P$ be a point in $\SK(\mathbb{F}_{q^4}) \setminus \SK(\mathbb{F}_q)$ with $(x,y,t)$-coordinates $(a,b,c)$. The aim of this section is to prove the following theorem.
\begin{theorem}
  \label{teo:HPgen}
  The Weierstrass semigroup $H(P)$ is generated by the following $3q_0+3$ elements:
  \begin{align*}
    & \mathrm{g}_0 := q^2-q+1, \\
    & \mathrm{g}_1 := q^2-2q_0+1, \\
    & \mathrm{g}_2 := q^2-q_0+1, \\
    & \mathrm{g}_3 := q^2, \\
    & \mathrm{g}_4 := q^2+1, \\
    & \mathrm{f}_i := (i+1)q_0g_0 - ig_4 - 1 \quad i=0,\dots,2q_0-2, \\
    & \mathrm{h}_j := (2j+1)q_0g_0 - jg_4 - q_0 \quad j=0,\dots,q_0-2.
  \end{align*}
\end{theorem}
The strategy of the proof will be to first show that the $3q_0+3$ elements mentioned in Theorem \ref{teo:HPgen} are in $H(P)$, then showing that the semigroup generated by them has $g(\SK)$ gaps.
As before we write $Q=\mathrm{pr}(P)$, the point of $\mathcal S$ lying under $P$ in the cover $\mathrm{pr}: \SK \to \mathcal S$. 
\begin{lemma}  \label{lem:PrincDiv}
There exists a divisor $D \geq 0$ on $\mathcal{S}$ satisfying $\Phi^i(Q) \not \in \Supp(D)$ for $i=0, \dots, 3$ such that:
  \begin{itemize}
  \item [a)] 
  The divisor $q_0 \sum_{i=0}^{3} \Phi^i(Q) + D - (q+2q_0+1)Q_\infty$ is principal.
  \item [b)] The divisor $((2i+1)qq_0+q_0)Q + (iq+q_0) \Phi(Q) + (q_0-1-i) \Phi^3(Q) + D - ((2i+1)q_0-i)(q+2q_0+1) Q_\infty$ is principal for all $i \geq 0$.
  \end{itemize}
\end{lemma}

\begin{proof}
a) Using the transitivity of the automorphism group of $\mathcal S$ on $\mathcal{S}(\mathbb{F}_{q^4})\setminus \mathcal{S}(\mathbb{F}_{q})$, we may assume without loss of generality that $a \in \mathbb{F}_{q^2} \setminus \mathbb{F}_q$. Define $\alpha := (a^q + a)^{q_0} \in \mathbb{F}_q$ and let
    \begin{equation}
      \label{eq:F0}
      F_0 := \alpha (a^{q_0}x + y + a^{q_0+1} + b) + x(x+a)^{2q_0} + (y+b)^{2q_0}.
    \end{equation}
    Observe that $F_0$ is defined over $\mathbb{F}_q$, since Equation \eqref{eq:F0} can also be expressed as
    $$ F_0 = a^{(q+1)q_0} x + \alpha y + y^{2q_0} + x^{2q_0+1} + (\alpha (a^{q_0+1}+b) + b^{2q_0}),$$
    where
    $(\alpha (a^{q_0+1}+b) + b^{2q_0})$ lies in $\mathbb{F}_q$ because
    $$ (\alpha (a^{q_0+1}+b) + b^{2q_0})^q + (\alpha (a^{q_0+1}+b) + b^{2q_0}) = \alpha (a^{(q_0+1)q} + a^{q_0+1}) + \alpha (b^q + b) + (b^q + b)^{2q_0} = $$
    $$ \alpha (a^{(q_0+1)q} + a^{q_0+1}) + \alpha a^{q_0} (a^q + a) + (a^{q_0} (a^q + a))^{2q_0} =
    \alpha^2 a^q + a^q \alpha^2 = 0. $$
    Recalling Equation \eqref{ytilda}, we have
    \begin{equation}
      \label{eq:vQF0}
      v_Q(F_0) \geq \min \{ v_Q(a^{q_0}x + y + a^{q_0+1} + b), v_Q(x(x+a)^{2q_0}), v_Q((y+b)^{2q_0}) \} = v_Q(\tilde{y}_Q) = q_0.
    \end{equation}
    Inequality \eqref{eq:vQF0} is in fact an equality, since $x(x+a)^{2q_0}$ and $(y+b)^{2q_0}$ have valuation at $Q$ larger than or equal to $2q_0$. The valuation of $F_0$ at $\Phi^i(Q)$, for $i=0, \dots, 3$ can be obtained from Inequality \eqref{eq:vQF0} and from the fact that $F_0$ is defined over $\mathbb{F}_q$, as
    $$ v_{\Phi^i(Q)}(F_0) = v_Q(\Phi^{-i}(F_0)) = v_Q(F_0) = q_0. $$
    Finally, observe that $Q_\infty$ is the only pole of $F_0$ and that
    $$ v_{Q_\infty}(F_0) = v_{Q_\infty}(x^{2q_0+1}+y^{2q_0}) = v_{Q_\infty}(z) = - (q+2q_0). $$
    Therefore there exists an effective divisor $\tilde{D}$ with support not containing $\Phi^i(Q)$ for $i = 0,\dots,3$ such that
    $$ (F_0)_{\mathcal S} = q_0 \sum_{i=0}^{3} \Phi^i(Q) + \tilde{D} - (q+2q_0)Q_\infty. $$
    The conclusion follows after setting $D := \tilde{D} + Q_\infty$.

b) Let us write $$D_i:=((2i+1)qq_0+q_0)Q + (iq+q_0) \Phi(Q) + (q_0-1-i) \Phi^3(Q) + D - ((2i+1)q_0-i)(q+2q_0+1) Q_\infty.$$
 Using part a) and Equation \eqref{wtilda}, we observe that $D_i = (F_0)_{\mathcal{S}} + (\tilde{w}_Q^{(2i+1)q_0})_{\mathcal{S}} - (\tilde{w}_{\Phi(Q)}^{i+1})_{\mathcal{S}}$.
\end{proof}

\begin{proposition}
  \label{lem:gensinHP}
The integers  $\mathrm{g}_0, \mathrm{g}_1, \mathrm{g}_2, \mathrm{g}_3, \mathrm{g}_4, \mathrm{f}_0, \dots, \mathrm{f}_{2q_0-2}, \mathrm{h}_0, \dots, \mathrm{h}_{q_0-2}$ are elements of $H(P)$.
\end{proposition}

\begin{proof}
  The first five values $\mathrm{g}_0, \mathrm{g}_1, \mathrm{g}_2, \mathrm{g}_3, \mathrm{g}_4$ can be obtained as pole orders at $P$ of the functions $\gamma_0 = \tilde{w}_Q \cdot \pi_P^{-1}$, $\gamma_1 = \tilde{w}_{\Phi^3(Q)} \cdot \pi_P^{-1}$, $\gamma_2 = \tilde{y}_Q \cdot \pi_P^{-1}$, $\gamma_3 = \tilde{x}_Q \cdot \pi_P^{-1}$, $\gamma_4 = \pi_P^{-1}$ respectively. It can be easily seen from Equations \eqref{ytilda}, \eqref{wtilda}, and \eqref{eq:div_pi} that such functions are regular outside $P$.

  Let us prove now that ${\mathrm{f}_i}$ is in $H(P)$ for $i=0, \dots, 2q_0-2$. For each $i=0, \dots, 2q_0-2$ define:
  $$ \alpha_i := \frac{\tilde{w}_Q^{(i+1)q_0} \cdot \tilde{w}_{\Phi^2(Q)}}{\tilde{w}_{\Phi(Q)}^{(i+1)}}. $$
  By Equation \eqref{wtilda}, the divisor of $\alpha_i$ in $\mathcal{S}$ is
  \begin{align*}
    (\alpha_i)_{\mathcal{S}} = ((i+1)qq_0+1) Q + (q-(i+1)q_0) \Phi^2(Q) + (2q_0-(i+1)) \Phi^3(Q) - ((i+1)q_0-i)(q+2q_0+1) Q_\infty.
  \end{align*}
  Note that the assumption $i \leq 2q_0-2$ implies that $(q-(i+1)q_0) \geq q_0 > 0$ and $(2q_0-(i+1)) \geq 1$. Since $P$ is unramified and $P_\infty$ is totally ramified in the cover $\mathrm{pr}: \SK \to \mathcal{S}$, we have:
  $$ \begin{cases}
    v_P(\alpha_i) = v_Q(\alpha_i) = (i+1)qq_0+1, & \\
    v_{P_\infty}(\alpha_i) = v_{Q_\infty}(\alpha_i) \cdot (q-2q_0+1) = - ((i+1)q_0-i)(q^2+1), & \\
    v_R(\alpha_i) \geq 0 & \quad \text{for } R \neq P, P_\infty.
  \end{cases} $$
  Therefore the function $\beta_i := \alpha_i / \pi_P^{(i+1)q_0-i}$ satisfies
  $$ \begin{cases}
    v_P(\beta_i) = ((i+1)qq_0+1) - ((i+1)q_0-i)(q^2+1) = -\mathrm{f}_i, & \\
    v_{P_\infty}(\beta_i) = - ((i+1)q_0-i)(q^2+1) + ((i+1)q_0-i)(q^2+1) = 0, & \\
    v_R(\beta_i) = v_R(\alpha_i) \geq 0 & \quad \text{for } R \neq P, P_\infty.
  \end{cases} $$
  Hence, $\beta_i$ is a function having a unique pole in $P$ of order $\mathrm{f}_i$.

  Let us finally prove that $\mathrm{h}_i$ is in $H(P)$ for $i=0, \dots, q_0-2$. By Lemma \ref{lem:PrincDiv} part b), there exists a function $\delta_i \in \mathbb{F}_{q^4}(\mathcal{S})$ and some effective divisor $D$ such that
  $$ (\delta_i)_{\mathcal S} = D_i := ((2i+1)qq_0+q_0)Q + (iq+q_0) \Phi(Q) + (q_0-1-i) \Phi^3(Q) + D - ((2i+1)q_0-i)(q+2q_0+1) Q_\infty. $$
  The assumption $i \leq q_0-2$ guarantees that $\delta_i$ is regular outside $Q_\infty$. Recalling Equation \eqref{eq:div_pi}, the function $\eta_i := \delta_i / \pi_P^{(2i+1)q_0-i}$ satisfies
  $$ \begin{cases}
    v_P(\eta_i) = ((2i+1)qq_0+q_0) - ((2i+1)q_0-i)(q^2+1) = -\mathrm{h}_i, & \\
    v_{P_\infty}(\eta_i) = - ((2i+1)q_0-i)(q^2+1) + ((2i+1)q_0-i)(q^2+1) = 0, & \\
    v_R(\eta_i) = v_R(\delta_i) \geq 0 & \quad \text{for } R \neq P, P_\infty.
  \end{cases} $$
  Hence, $\eta_i$ is a function having a unique pole in $P$ of order $\mathrm{h}_i$.
\end{proof}

A consequence of Proposition \ref{lem:gensinHP} is that the numerical semigroup
$$ H := \langle \mathrm{g}_0, \mathrm{g}_1, \mathrm{g}_2, \mathrm{g}_3, \mathrm{g}_4, \mathrm{f}_i, \mathrm{h}_j \mid i=0,\dots,2q_0-2 \text{ and } j=0,\dots,q_0-2 \rangle $$
is contained in $H(P)$. To show that equality $H = H(P)$ holds, from which Theorem \ref{teo:HPgen} follows, it suffices to prove that $H$ has the same genus of $H(P)$, namely that $g(H) = g(H(P)) = \frac{1}{2}q(q-1)^2$.

Observe that $\mathrm{g}_0$ is the smallest generator of $H$; in particular $\mathrm{g}_0$ is the multiplicity of $H$ and the Ap\'ery set $\mathrm{Ap}(H)$ consists of $\mathrm{g}_0$ distinct elements. We now want to find a map $\varphi : \{ 0,\dots,\mathrm{g}_0-1 \} \to \mathbb{N}$ such that any $a \in \mathrm{Ap}(H)$ can be expressed as $a = \varphi(i) \mathrm{g}_0 + i$ for some $i \in \{ 0,\dots,\mathrm{g}_0-1 \}$.

\begin{lemma}
  \label{lem:phi1}
  Let $i \in \{ 0, \dots, \frac{q(q-2)}{2} \}$ and write $i = lq+kq_0+j$ with $j \in \{ 0,\dots,q_0-1 \}$, $k \in \{ 0,\dots,2q_0-1 \}$ and $l \geq 0$. Define:
  $$ \varphi_1(i) =
  \begin{cases}
    l & \quad \text{if } j = 0 \text{ and } k = 0, \\
    l+1 + \max \{ q - q_0(k + 2l + 2), 0 \} & \quad \text{if } j = 0 \text{ and } k \neq 0, \\
    l+1 + \max \{ q - q_0(\lceil \frac{k}{2} \rceil + j + l + 1), 0 \} & \quad \text{if } j \neq 0.
  \end{cases}
  $$
  Then $\varphi_1(i)\mathrm{g}_0 + i$ belongs to $H$.
\end{lemma}

\begin{proof}
  Clearly, if $j=0$ and $k=0$, then $\varphi_1(i)\mathrm{g}_0 + i = l\mathrm{g}_0 + lq = l\mathrm{g}_4$ belongs to $H$.

  We now consider the case of $j = 0, k \neq 0$ and $\max \{q - q_0(k + 2l + 2), 0 \} > 0$. Note that this last condition implies that $2q_0 - k - 2l - 2 > 0$. Then:
  \begin{align*}
    \varphi_1(i)\mathrm{g}_0+i & = (l+1 + q - q_0(k + 2l + 2))\mathrm{g}_0 + lq + kq_0 \\
                      & = (l+1)\mathrm{g}_1 + (l+1)(2q_0-q) + (q-q_0(k + 2l + 2))\mathrm{g}_0 + lq + kq_0 \\
                      & = (l+1)\mathrm{g}_1 + (2q_0 - k - 2l - 2)\mathrm{h}_0.
  \end{align*}
  Hence $\varphi_1(i)\mathrm{g}_0+i$ belongs to $H$.

  Let us consider now the case of $j = 0, k \neq 0$ and $\max \{q - q_0(k + 2l + 2), 0 \} = 0$. This last condition can be expressed equivalently as $k \geq 2q_0-2l-2$. Define $c := 2q_0-k$, so that $1 \leq c \leq 2l+2$. Then
  \begin{equation}
    \label{eq:j0}
    \varphi_1(i)\mathrm{g}_0+i = (l+1)\mathrm{g}_0 + lq + kq_0 = (l+1)\mathrm{g}_4 - cq_0.
  \end{equation}
  We can distinguish three cases.
  \begin{itemize}
  \item If $c \leq l+1$, then Equation \eqref{eq:j0} can be written as $ (l+1-c)\mathrm{g}_4 + c\mathrm{g}_2. $
  \item If $c > l+1$ and $c$ is even, then Equation \eqref{eq:j0} can be written as $ \left( l+1-\frac{c}{2} \right) \mathrm{g}_4 + \frac{c}{2} \mathrm{g}_1. $
  \item If $c > l+1$ and $c$ is odd (and in particular $c \leq 2l + 1$), then Equation \eqref{eq:j0} can be written as $ \left( l-\frac{c-1}{2} \right) \mathrm{g}_4 + \frac{c-1}{2} \mathrm{g}_1 + \mathrm{g}_2. $
  \end{itemize}
  In all the three cases $\varphi_1(i)\mathrm{g}_0+i$ belongs to $H$.

  Let us now assume $j \neq 0$ and $\max \{q - q_0(\lceil \frac{k}{2} \rceil + j + l + 1), 0 \} > 0$; this last condition implies $2q_0 - \lceil \frac{k}{2} \rceil - j - l - 1 > 0$. Also,
  \begin{align}
    \label{eq:jneq0}
    \varphi_1(i)\mathrm{g}_0+i & = \left( l+1 + q - q_0 \left( \left\lceil \frac{k}{2} \right\rceil + j + l + 1 \right) \right) \mathrm{g}_0 + lq + kq_0 + j \nonumber \\
                      & = \mathrm{f}_{2q_0 - \lceil \frac{k}{2} \rceil - j - l - 2} + 2q_0q^2 + 2q_0 - \left\lceil \frac{k}{2} \right\rceil (q^2+1) - j q^2 - q - q^2 + kq_0.
  \end{align}
  If $k$ is even, then Equation \eqref{eq:jneq0} is equal to
  $$ \mathrm{f}_{2q_0 - \frac{k}{2} - j - l - 2} + \left( q_0 - 1 - \frac{k}{2} \right) \mathrm{g}_1 + (q_0 - 1 - j) \mathrm{g}_3 + \mathrm{g}_2. $$
  If $k$ is odd, then Equation \eqref{eq:jneq0} is equal to
  $$ \mathrm{f}_{2q_0 - \frac{k+1}{2} - j - l - 2} + \left( q_0 - \frac{k+1}{2} \right) \mathrm{g}_1 + (q_0 - 1 - j) \mathrm{g}_3. $$
  In both cases $\varphi_1(i)\mathrm{g}_0+i$ belongs to $H$.

  Let us finally consider the case of $j \neq 0$ and $\max \{q - q_0(\lceil \frac{k}{2} \rceil + j + l + 1), 0 \} = 0$; this last condition implies $k \geq 4q_0-2j-2l-3$. Define $c := 2q_0-1-k$, so that $0 \leq c \leq 2(l+1+j-q_0)$. Then
  \begin{align}
    \label{eq:jneq0max0}
    \varphi_1(i)\mathrm{g}_0+i & = (l+1)\mathrm{g}_0 + lq + kq_0 + j
                       = (l+1+j-q_0)\mathrm{g}_4 + (q_0-j)\mathrm{g}_3 - cq_0.
  \end{align}
  Let us distinguish three cases.
  \begin{itemize}
  \item If $c \leq l+1+j-q_0$, then Equation \eqref{eq:jneq0max0} can be written as $ (l+1+j-q_0-c)\mathrm{g}_4 + (q_0-j)\mathrm{g}_3 + c\mathrm{g}_2. $
  \item If $c > l+1+j-q_0$ and $c$ is even, then Equation \eqref{eq:jneq0max0} can be written as
    $$ \left( l+1+j-q_0-\frac{c}{2} \right) \mathrm{g}_4 + (q_0-j)\mathrm{g}_3 + \frac{c}{2} \mathrm{g}_1.$$
  \item If $c > l+1+j-q_0$ and $c$ is odd (and in particular $1 \leq c \leq 1+2(l+j-q_0)$), then Equation \eqref{eq:jneq0max0} can be written as
    $$ \left( l+j-q_0-\frac{c-1}{2} \right) \mathrm{g}_4 + (q_0-j)\mathrm{g}_3 + \frac{c-1}{2} \mathrm{g}_1 + \mathrm{g}_2. $$
  \end{itemize}
  In all the three cases $\varphi_1(i)\mathrm{g}_0+i$ belongs to $H$.
\end{proof}

\begin{lemma}
  \label{lem:phi2}
  Let $i \in \{ \frac{q(q-2)}{2} + 1, \dots, \mathrm{g}_0-1 \}$. Denote $i^\prime = \mathrm{g}_0-1-i$ and write $i^\prime = lq+kq_0+j$ with $j \in \{ 0,\dots,q_0-1 \}$, $k \in \{ 0,\dots,2q_0-1 \}$ and $l \geq 0$. Define:
  $$ \varphi_2(i) =
  \begin{cases}
    q-l-1 - \max \{q - q_0(k + 2l + 1), 0 \} & \quad \text{if } j = q_0-1, \\
    q-l-1 - \max \{q - q_0(\lceil \frac{k}{2} \rceil + j + l + 1), 0 \} & \quad \text{if } j \neq q_0-1.
  \end{cases}
  $$
  Then $\varphi_2(i)\mathrm{g}_0 + i$ belongs to $H$.
\end{lemma}

\begin{proof}
  We first prove that if $\varphi_2(i) = q-l-1$, then $\varphi_2(i)g\mathrm{g}_0+i$ belongs to $H$.
  \begin{align*}
    \varphi_2(i)\mathrm{g}_0+i & = (q-l-1)\mathrm{g}_0 + \mathrm{g}_0 - 1 - i^\prime \\
                      & = (q-l-1)\mathrm{g}_4 + lq - (lq+kq_0+j) \\
                      & = (q-l-k-j-1)\mathrm{g}_4 + k\mathrm{g}_2 + j\mathrm{g}_3.
  \end{align*}
  By assumption on $i$, the condition $l \leq \frac{q}{2}-1$ holds, so
  \begin{align*}
    q-l-k-j-1 & \geq q - \left( \frac{q}{2}-1 \right) - (2q_0-1) - (q_0-1) - 1 \\
              & = \frac{q}{2} - 3q_0 + 2
               = (q_0-1)(q_0-2) \geq 0.
  \end{align*}
  Therefore $\varphi_2(i)\mathrm{g}_0+i$ is in $H$.

  We now consider the case of $j = q_0-1$ and $\max \{q - q_0(k + 2l + 1), 0 \} > 0$. Note that this last condition implies $l < q_0$. Define $\mathrm{h}_{q_0-1} := (2q_0-1)q_0\mathrm{g}_0 - (q_0-1)\mathrm{g}_4 - q_0$, coherently with the definition of $\mathrm{h}_j$ for $j=0,\dots,q_0-2$, and observe that $\mathrm{h}_{q_0-1} = \mathrm{f}_{q_0-1} + \mathrm{f}_{q_0-2} + (q_0-2)\mathrm{g}_3 \in H$. Then:
  \begin{align*}
    \varphi_2(i)\mathrm{g}_0+i & = (q_0(k + 2l + 1) - l - 1)\mathrm{g}_0 + \mathrm{g}_0 - 1 - i^\prime \\
                    & = (q_0(k + 2l + 1) - l)\mathrm{g}_0 - 1 - (lq+kq_0+j) \\
                    & = \mathrm{h}_l + k\mathrm{h}_0.
  \end{align*}
  Hence $\varphi_2(i)\mathrm{g}_0+i$ belongs to $H$.

  We finally prove that if $j \neq q_0-1$ and $\max \{q - q_0(\lceil \frac{k}{2} \rceil + j + l + 1), 0 \} > 0$, then $\varphi_2(i)\mathrm{g}_0+i$ belongs to $H$. Note that the condition on the maximum implies $\lceil \frac{k}{2} \rceil + j + l < 2q_0 - 1$. Then:
  \begin{align}
    \label{eq:jneqq0-1}
    \varphi_2(i)\mathrm{g}_0+i & = \left( q_0 \left( \left\lceil \frac{k}{2} \right\rceil + j + l + 1 \right) - l - 1 \right) \mathrm{g}_0 + \mathrm{g}_0-1-i^\prime \nonumber \\
                    & = \left( q_0 \left( \left\lceil \frac{k}{2} \right\rceil + j + l + 1 \right) - l \right) \mathrm{g}_0 - 1 - (lq+kq_0+j) \nonumber \\
                    & = \mathrm{f}_{\lceil \frac{k}{2} \rceil + j + l} + \left( \left\lceil \frac{k}{2} \right\rceil + j \right) q^2 - kq_0 + \left\lceil \frac{k}{2} \right\rceil.
  \end{align}
  If $k$ is even, then Equation \eqref{eq:jneqq0-1} is equal to $\mathrm{f}_{\frac{k}{2} + j + l} + \frac{k}{2}\mathrm{g}_1 + j\mathrm{g}_3. $
If $k$ is odd (and in particular $k \geq 1$), then Equation \eqref{eq:jneqq0-1} is equal to
    $ \mathrm{f}_{\frac{k+1}{2} + j + l} + \frac{k-1}{2}\mathrm{g}_1 + j\mathrm{g}_3 + \mathrm{g}_2. $
In both cases $\varphi_2(i)\mathrm{g}_0+i$ belongs to $H$.
\end{proof}

Define the map $\varphi : \{ 0, \dots, \mathrm{g}_0-1 \} \to \mathbb{N}$ by
$$ \varphi(i) =
\begin{cases}
  \varphi_1(i) & \quad \text{if } i \in \{ 0, \dots, \frac{q(q-2)}{2} \}, \\
  \varphi_2(i) & \quad \text{if } i \in \{ \frac{q(q-2)}{2} + 1, \dots, \mathrm{g}_0-1 \}.
\end{cases}
$$
Combining Lemma \ref{lem:phi1} and Lemma \ref{lem:phi2}, it follows that $\varphi(i)\mathrm{g}_0 + i$ belongs to $H$ for all $i \in \{0, \dots, \mathrm{g}_0-1 \}$.

\begin{lemma}
  \label{lem:sumeqg}
The Ap\'ery set of $H(P)$ is $\mathrm{Ap}(H(P)) = \{ \varphi(i)\mathrm{g}_0 + i \mid i=0, \dots, \mathrm{g}_0-1 \}$.
\end{lemma}

\begin{proof}
It is sufficient to show that $\sum_{i=0}^{\mathrm{g}_0-1} \varphi(i) = g(\SK)$. Indeed, the number of gaps of $H$ is then shown to be at most $g(\SK)$, which implies that $H=H(P),$ since we already know that $H \subseteq H(P).$ But in that case, the mentioned set is the Ap\'ery set of $H(P)$, since otherwise $g(H(P))$ would be strictly less than $\sum_{i=0}^{\mathrm{g}_0-1} \varphi(i)$.

First, let us prove that
  \begin{equation}
    \label{eq:phi1+phi2}
    \varphi(i) + \varphi((q-1)^2-i) = q-1 \quad \text{for } i \in \{ 0,\dots,(q-1)^2 \}.
  \end{equation}
  Observe that $i$ is in range $0,\dots,\frac{q(q-2)}{2}$ if and only if $(q-1)^2-i$ is in range $\frac{q(q-2)}{2}+1,\dots,(q-1)^2$. Therefore it is enough to prove Equation \eqref{eq:phi1+phi2} for $i \in \{ 0,\dots,\frac{q(q-2)}{2} \}$ only; the case $i \in \{ \frac{q(q-2)}{2}+1,\dots,(q-1)^2 \}$ follows by symmetry.

  For $i \in \{ 0,\dots,\frac{q(q-2)}{2} \}$ we have $\varphi(i) = \varphi_1(i)$ where $\varphi_1$ is defined as in Lemma \ref{lem:phi1}  and $\varphi((q-1)^2-i) = \varphi_2((q-1)^2-i)$ where $\varphi_2$ is defined as in Lemma \ref{lem:phi2}. Write $i = lq+kq_0+j$ with $j \in \{ 0,\dots,q_0-1 \}$, $k \in \{ 0,\dots,2q_0-1 \}$ and $l \geq 0$. Then:
  $$ \mathrm{g}_0-1-((q-1)^2-i) = q-1+i = l^\prime q + k^\prime q_0 + j^\prime $$
  with
  $$ \begin{cases}
    l^\prime = l, k^\prime = 2q_0-1, j^\prime = q_0-1 & \quad \text{if } j=0 \text{ and } k=0, \\
    l^\prime = l+1, k^\prime = k-1, j^\prime = q_0-1 & \quad \text{if } j=0 \text{ and } k \neq 0, \\
    l^\prime = l+1, k^\prime = k, j^\prime = j-1 & \quad \text{if } j \neq 0,
  \end{cases} $$
  satisfying $j^\prime \in \{ 0,\dots,q_0-1 \}$, $k^\prime \in \{ 0,\dots,2q_0-1 \}$ and $l^\prime \geq 0$. Consequently:
  \begin{align*}
    & q-1-\varphi((q-1)^2-i) = q-1-\varphi_2((q-1)^2-i) \\
    & = \begin{cases}
      l^\prime + \max \{q - q_0(k^\prime + 2l^\prime + 1), 0 \} & \quad \text{if } j^\prime = q_0-1, \\
      l^\prime + \max \{q - q_0(\lceil \frac{k^\prime}{2} \rceil + j^\prime + l^\prime + 1), 0 \} & \quad \text{if } j^\prime \neq q_0-1.
    \end{cases} \\
    & = \begin{cases}
      l + \max \{ -2q_0l, 0 \} & \quad \text{if } j=0 \text{ and } k=0, \\
      l+1 + \max \{ q - q_0(k+2l+2), 0 \} & \quad \text{if } j=0 \text{ and } k \neq 0, \\
      l+1 + \max \{ q - q_0( \left\lceil \frac{k}{2} \right\rceil + j + l + 1), 0 \} & \quad \text{if } j \neq 0.
    \end{cases} \\
    & = \varphi_1(i) = \varphi(i).
  \end{align*}
  Equation \eqref{eq:phi1+phi2} implies that
  \begin{equation}
    \label{eq:partialsum1}
    \sum_{i=0}^{(q-1)^2} \varphi(i) = \sum_{i=0}^{\frac{q(q-2)}{2}} (q-1) = \left( \frac{q(q-2)}{2} + 1 \right) (q-1).
  \end{equation}
  Now let us assume $i \in \{ (q-1)^2+1, \dots, \mathrm{g}_0-1 \}$. Since $i^\prime = \mathrm{g}_0-1-i$ ranges between $0$ and $q-2$, we can express $i^\prime$ as $i^\prime = kq_0+j$ with $j \in \{ 0,\dots,q_0-1 \}$ and $k \in \{ 0,\dots,2q_0-1 \}$. In particular:
  $$ \varphi(i) = \varphi_2(i) =
  \begin{cases}
    q_0(k+1) - 1 & \quad \text{if } j = q_0-1, \\
    q_0( \left\lceil \frac{k}{2} \right\rceil + j + 1) - 1 & \quad \text{if } j \neq q_0-1.
  \end{cases} $$
  Then:
  \begin{align}
    \sum_{i=(q-1)^2+1}^{\mathrm{g}_0-1} \varphi(i) & = \sum_{k=0}^{2q_0-2} (q_0(k+1)-1) + \sum_{k=0}^{2q_0-1} \sum_{j=0}^{q_0-2} \left( q_0 \left( \left\lceil \frac{k}{2} \right\rceil + j+1 \right) -1 \right)
                                           = \frac{q^2}{2} - \frac{3}{2}q + 1. \label{eq:partialsum2}
  \end{align}
  Combining Equations \eqref{eq:partialsum1} and \eqref{eq:partialsum2} we obtain:
  $$ \sum_{i=0}^{\mathrm{g}_0-1} \varphi(i) = \left( \frac{q(q-2)}{2} + 1 \right) (q-1) + \left( \frac{q^2}{2} - \frac{3}{2}q + 1 \right) = \frac{1}{2}q(q-1)^2 = g(\SK). $$
\end{proof}
Note that from the proof of the above lemma, we immediately conclude that $H=H(P)$, proving Theorem \ref{teo:HPgen}.

\begin{corollary}
Let $P \in \SK(\mathbb{F}_{q^4})\setminus \SK(\mathbb{F}_{q})$. Then the Weierstrass semigroup $H(P)$ is symmetric.
\end{corollary}

\begin{proof}
  For $i = (q-1)^2$ we have $a := \varphi(i)\mathrm{g}_0 + i = q^2(q-1) = 2g(\SK)-1+\mathrm{g}_0$. Since $a$ is an element of $\mathrm{Ap}(H(P))$, it follows that $a - \mathrm{g}_0 = 2g(\SK)-1$ is a gap of $H$. By the Weierstrass gap theorem $2g(\SK)-1$ is the largest possible gap of $H(P)$, hence $H(P)$ has conductor $2g(\SK)$.
\end{proof}

As for $P \in \SK(\mathbb{F}_q)$, we can now conclude from \cite[Proposition 50]{Konto} that any $P \in \SK(\mathbb{F}_{q^4})\setminus \SK(\mathbb{F}_{q})$ is a Weierstrass point of $\tilde{\mathcal{S}}$. In the next section, we will show that any $P \not \in \SK(\mathbb{F}_{q^4})$ has the same Weierstrass semigroup. This will imply that the set of Weierstrass points of $\SK$ is equal to $\SK(\mathbb{F}_{q^4}).$

\section{The Weierstrass semigroup for $P \not \in \SK(\mathbb{F}_{q^4})$}\label{sec:five}

Let now $P \in \SK(\overline{\mathbb{F}}_{q}) \setminus \SK(\mathbb{F}_{q^4})$. Let $a,b,c \in \overline{\mathbb{F}}_{q}$ be the affine coordinates of $P$ and as before denote with $Q$, the point of $\mathcal S$ with affine coordinates $a$ and $b$. 
It will also be convenient to use the expression $A:=a^q+a$.

From Corollary \ref{basisdiff}, the gap sequence $G(P)$ at $P$ can be computed by constructing $g=g(\SK)$ functions $f_1,\ldots,f_g$ having pairwise distinct valuations at $P$ and such that $f_i \in L((2g(\SK)-2)P_\infty)$ for all $i=1,\ldots,g$.

To this aim note first that $\tilde{x}_Q:=x-a$ is a local parameter at $P$ and $\tilde{x}_Q,\tilde{y}_Q,\tilde{z}_Q,\tilde{w}_Q \in L((q^3-2q^2+q-2)P_\infty)$ by Equations \eqref{xtilda}-\eqref{wtilda}. 
%
%
Also, recall from Equation \eqref{eq:div_pi}, that there exists a function $\pi_P \in \overline{\mathbb{F}}_{q}(\SK)$ such that $(\pi_P)_{\SK} = q^2P+\Phi^4(P)-(q^2+1)P_\infty$.

In the following, the local power series expansions of $\tilde{y}_Q,\tilde{z}_Q$ and $\tilde{w}_Q$ at $P$ (with respect to the local parameter $\tilde{x}_Q$) is computed. First of all note that from $Q \in \mathcal{S}$ and $y^q+y=x^{q_0}(x^q+x)$ we have
$$(y+b)^q+(y+b)=a^{q_0}(a^q+a)+x^{q_0}(x^q+x)=a^{q_0}\tilde{x}_Q+(a^q+a)\tilde{x}_Q^{q_0}+\tilde{x}_Q^{q_0+1}+a^{q_0}\tilde{x}_Q^q+\tilde{x}_Q^{q+q_0},$$
so that
\begin{equation} \label{powerseriesyb}
(y+b)=a^{q_0} \tilde{x}_Q +A\tilde{x}_Q^{q_0}+\tilde{x}_Q^{q_0+1}+A^{q_0}\tilde{x}_Q^q+\tilde{x}_Q^{q+q_0}+A^q\tilde{x}_Q^{q_0 q}+\tilde{x}_Q^{qq_0+q}+A^{q_0q}+\tilde{x}_Q^{q^2}+O(\tilde{x}_Q^{q^2+1}).
\end{equation}
Hence, from $\tilde{y}_Q=(y+b)+a^{q_0}\tilde{x}_Q$ and Equation \eqref{powerseriesyb} we get
\begin{equation} \label{powerseriesytilda}
\tilde{y}_Q=A\tilde{x}_Q^{q_0}+\tilde{x}_Q^{q_0+1}+A^{q_0}\tilde{x}_Q^q+\tilde{x}_Q^{q+q_0}+A^q\tilde{x}_Q^{q_0 q}+\tilde{x}_Q^{qq_0+q}+A^{q_0q}+\tilde{x}_Q^{q^2}+O(\tilde{x}_Q^{q^2+1}).
\end{equation}
Since $A \ne 0$ from $P \not\in \SK(\mathbb{F}_q)$, while $\tilde{x}_Q$, seen as a function on $\mathcal{S}$, is also a local parameter at $Q$, we get that $v_P(\tilde{y}_Q)=v_Q(\tilde{y}_Q)=q_0$ as anticipated in Equation \eqref{ytilda}.

Now, $\tilde{z}_Q=a^{2q_0}(\tilde{x}_Q+a)+(\tilde{x}_Q+a)^{2q_0+1}+(y+b)^{2q_0}=a\tilde{x}_Q^{2q_0}+\tilde{x}_Q^{2q_0+1}+(y+b)^{2q_0}$, which combined with Equation \eqref{powerseriesyb} gives
\begin{equation} \label{powerseriesztilda}
\tilde{z}_Q=A\tilde{x}_Q^{2q_0}+\tilde{x}_Q^{2q_0+1}+A^{2q_0}\tilde{x}_Q^{q}+\tilde{x}_Q^{q+2q_0}+A^{q}\tilde{x}_Q^{2q_0q}+\tilde{x}_Q^{2q_0q+q}+A^{2q_0q}\tilde{x}_Q^{q^2}+O(\tilde{x}_Q^{q^2+1}).
\end{equation}
The computation above yields that $v_P(\tilde{z}_Q)=v_Q(\tilde{z}_Q)=2q_0$, which is again consistent with Equation \eqref{ztilda}.
Finally, from $z=\tilde{z}_Q+a^{2q_0}x+b^{2q_0}$, $b^q+b=a^{q_0}(a^q+a)$, Equations \eqref{powerseriesyb} and \eqref{powerseriesztilda}, we have
\begin{eqnarray*}
\tilde{w}_Q&=& a^q \tilde{z}_Q+b^{2q_0}(\tilde{x}_Q+a)+(\tilde{z}_Q+a^{2q_0}x+b^{2q_0})^{2q_0}+xy^{2q_0}+b^2+a^{2q_0+2}\\
&=&  a^q \tilde{z}_Q+\tilde{z}_Q^{2q_0}+a^{2q} \tilde{x}_Q^{2q_0}+\tilde{x}_Q(y+b)^{2q_0}+a(y+b)^{2q_0}\\
&=& a^q(A\tilde{x}_Q^{2q_0}+\tilde{x}_Q^{2q_0+1}+A^{2q_0}\tilde{x}_Q^{q}+\tilde{x}_Q^{q+2q_0}+A^{q}\tilde{x}_Q^{2q_0q}+\tilde{x}_Q^{2q_0q+q}+A^{2q_0q}\tilde{x}_Q^{q^2})\\
&& +(A^{2q_0}\tilde{x}_Q^{2q}+\tilde{x}_Q^{2q+2q_0}+A^{2q}\tilde{x}_Q^{2q_0q}+\tilde{x}_Q^{2q_0q+2q})\\
&&+a^{2q} \tilde{x}_Q^{2q_0}+\tilde{x}_Q( a^{q} \tilde{x}_Q^{2q_0} +A^{2q_0}\tilde{x}_Q^{q}+\tilde{x}_Q^{q+2q_0}+A^{q}\tilde{x}_Q^{2q_0q}+\tilde{x}_Q^{2q_0q+q})\\
&&+a( a^{q} \tilde{x}_Q^{2q_0} +A^{2q_0}\tilde{x}_Q^{q}+\tilde{x}_Q^{q+2q_0}+A^{q}\tilde{x}_Q^{2q_0q}+\tilde{x}_Q^{2q_0q+q}+A^{2q_0q}\tilde{x}_Q^{q^2})+O(\tilde{x}_Q^{q^2+1}),
\end{eqnarray*}
which yields,
\begin{eqnarray}
 \label{powerserieswtilda}
\tilde{w}_Q&=& A^{2q_0+1}\tilde{x}_Q^q+A^{2q_0}\tilde{x}_Q^{q+1}+A\tilde{x}_Q^{q+2q_0}+\tilde{x}_Q^{q+2q_0+1}+A^{2q_0}\tilde{x}_Q^{2q}+\tilde{x}_Q^{2q+2q_0}+A^q(A^q+A)\tilde{x}_Q^{2q_0q} \nonumber\\
&&+A^q\tilde{x}_Q^{2q_0q+1}+A\tilde{x}_Q^{2q_0q+q}+\tilde{x}_Q^{2q_0q+q+1}+\tilde{x}_Q^{2q_0q+2q}+A^{2q_0q+1}\tilde{x}_Q^{q^2}+O(\tilde{x}_Q^{q^2+1}).
\end{eqnarray}
Hence $v_P(\tilde{w}_Q)=v_Q(\tilde{w}_Q)=q$ as in Equation \eqref{wtilda}.

%

\begin{rem} \label{remFrob}
Fixing $i \geq 1$ and replacing $P$ with $\Phi^i (P)$, one can obtain the power series expansions of $\tilde{x}_{\Phi^i(Q)}$, $\tilde{y}_{\Phi^i(Q)}$, $\tilde{z}_{\Phi^i(Q)}$ and $\tilde{w}_{\Phi^i(Q)}$. Clearly, they will be given as in Equations \eqref{powerseriesytilda}-\eqref{powerserieswtilda} simply replacing $\tilde{x}_Q$ with $\tilde{x}_{\Phi^i(Q)}$ and $A$ with $A^{q^i}$.
\end{rem}

The following lemmas show that the computations carried out in Equations \eqref{powerseriesytilda}-\eqref{powerserieswtilda} for $P$ (and for $\Phi^i(P)$ with $i \geq 1$ as in Remark \ref{remFrob}), as well as the divisor in Equation \eqref{wtilda} allow us to construct other functions in $L((2g(\SK)-2)P_\infty)$ that we can use to obtain families of gaps according to Corollary \ref{basisdiff}.

\begin{lemma} \label{functhn}
Let $n=1,\ldots,2q_0-2$. Then there exists a function $h_n$ on $\mathcal{S}$ that is regular outside $Q_\infty$ and such that
$$(h_n)_{\mathcal{S}}=(n+1)q_0qQ+q_0(2q_0-n-1)\Phi^2(Q)+(2q_0-n-1)\Phi^3(Q)+\Phi^4(Q)-[(n+1)q_0-n](q+2q_0+1)Q_\infty.$$
In particular $v_P(h_n)=(n+1)q_0q$, $v_{P_\infty}(h_n)=-[(n+1)q_0-n](q^2+1)$ and $h_n \in L((2g(\SK)-2)P_\infty)$.
\end{lemma}

\begin{proof}
It is sufficient to define
$$h_n=\tilde{w}_{\Phi^2(Q)} \cdot \bigg(\frac{\tilde{w}_Q^{q_0}}{\tilde{w}_{\Phi(Q)}}\bigg)^{n+1}.$$
Indeed from Equation \eqref{wtilda} we get
\[
(h_n)_{\mathcal{S}}= (n+1)q_0qQ+q_0(2q_0-n-1)\Phi^2(Q)+(2q_0-n-1)\Phi^3(Q)+\Phi^4(Q)-[(n+1)q_0-n](q+2q_0+1)Q_\infty.
\]
The claim follows by observing that $2q_0-n-1>0$ since $n \leq 2q_0-2$.
\end{proof}

\begin{lemma} \label{firstpair}
There exists a function $f_1$ on $\mathcal{S}$ that is regular outside $Q_\infty$ and such that
$$v_Q(f_1)=q_0q+q_0, \quad \textrm{and} \quad v_{Q_\infty}(f_1) \geq -q_0(q+2q_0+1).$$
In particular $v_P(f_1)=q_0q+q_0$, $v_{P_\infty}(f_1) \geq -q_0(q^2+1)$ and $f_1 \in L((2g(\SK)-2)P_\infty)$.
\end{lemma}

\begin{proof}
We have the following linear equivalence of divisors on $\mathcal{S}$,
\begin{eqnarray*}
(q_0q+q_0)Q-q_0(q+2q_0+1)Q_\infty &\cong& (q_0q+q_0)Q-q_0(q+2q_0+1)Q_\infty -(\tilde{w}_Q^{q_0})_{\mathcal{S}}+(\tilde{w}_{\Phi(Q)})_{\mathcal{S}}\\
&=& q_0 Q+ q_0 \Phi^2(Q)+\Phi^3(Q)-(q+2q_0+1)Q_\infty.
\end{eqnarray*}
Hence if we find a function $\tilde{f}_1$ such that $v_Q(\tilde{f}_1)=q_0$, $v_{\Phi^2(Q)}(\tilde{f}_1) \geq q_0$, $v_{\Phi^3(Q)}(\tilde{f}_1) \geq 1$ and $v_{Q_\infty}(\tilde{f}_1) \geq -(q+2q_0+1)$ then we can simply define
$f_1=\tilde{f}_1 \cdot \frac{\tilde{w}_Q^{q_0}}{\tilde{w}_{\Phi(Q)}},$
to complete the proof.

To construct $\tilde{f}_1$ we first observe that if $f_{\alpha,\beta,\gamma}$ denotes a linear combination $f_{\alpha,\beta,\gamma}:=\alpha  \tilde{w}_{\Phi(Q)}+\beta \tilde{y}_{\Phi^2(Q)}+ \gamma \tilde{w}_{\Phi^2(Q)}$ with $\alpha,\beta,\gamma \in \overline{\mathbb{F}}_{q^4}$, then

$$v_{\Phi^2(Q)}(f_{\alpha,\beta,\gamma}) \geq \min\{v_{\Phi^2(Q)}(\tilde{y}_{\Phi^2(Q)}),v_{\Phi^2(Q)}( \tilde{w}_{\Phi^2(Q)}), v_{\Phi^2(Q)}( \tilde{w}_{\Phi(Q)})\}=\min\{q_0,q,2q_0\}=q_0$$
and similarly,
$$v_{\Phi^3(Q)}(f_{\alpha,\beta,\gamma}) \geq \min\{v_{\Phi^3(Q)}(\tilde{y}_{\Phi^2(Q)}),v_{\Phi^3(Q)}( \tilde{w}_{\Phi^2(Q)}),v_{\Phi^3(Q)}( \tilde{w}_{\Phi(Q)})\}=\min\{1,2q_0,1\}=1.$$
Also, from Equations \eqref{ytilda} and \eqref{wtilda}, $f_{\alpha,\beta,\gamma}$ is regular outside $Q_\infty$ and $v_{Q_\infty}(f_{\alpha,\beta,\gamma}) \geq -(q+2q_0+1)$.
Hence we give a local description of the functions $\tilde{y}_{\Phi^2(Q)}$, $\tilde{w}_{\Phi^2(Q)}$ and $\tilde{w}_{\Phi(Q)}$  at $Q$ (equivalently $P$) and check whether we can find a linear combination $f_{\tilde \alpha,\tilde \beta, \tilde \gamma}$ such that $v_Q(f_{\tilde \alpha, \tilde \beta, \tilde \gamma}) =q_0$. Doing so, it will be enough to define $\tilde f_1:=f_{\tilde \alpha, \tilde \beta,\tilde \gamma}$. From Equation \eqref{powerseriesyb} we have that
\begin{eqnarray} \label{powerseriesytildaPhi2P}
\tilde{y}_{\Phi^2(Q)}&=&y+b^{q^2}+a^{q_0q^2}(x+a^{q^2}) =(y+b)+(b^{q^2}+b)+a^{q_0q^2}(\tilde{x}_Q+a+a^{q^2}) \nonumber\\
&=& (a^{q_0} \tilde{x}_Q+A\tilde{x}_Q^{q_0}+O(\tilde{x}_Q^{q_0+1}))+(b^q+b)+(b^q+b)^q+a^{q_0q^2}(\tilde{x}_Q+a+a^{q^2}) \nonumber \\
&=& T(A)+(A^q+A)^{q_0}\tilde{x}_Q+A\tilde{x}_Q^{q_0}+O(\tilde{x}_Q^{q_0+1}),
\end{eqnarray}
where using $a^{q_0}A=b^q+b$,
$$T(A)=(b^q+b)+(b^q+b)^q+a^{q_0q^2}(A^q+A)=a^{q_0}A+a^{q_0q}A^q+a^{q_0q^2}(A^q+A)=A^{q_0+1}+A^{qq_0+1}+A^{qq_0+q}.$$

From Equation \eqref{ztilda}, $z=\tilde{z}_Q+a^{2q_0}(\tilde{x}_Q+a)+b^{2q_0}$ and this combined with Equation \eqref{powerseriesztilda} gives
\begin{eqnarray} \label{tildazPhi2P}
\tilde{z}_{\Phi^2(Q)}&=& a^{2q_0q^2}(\tilde{x}_Q+a)+z+b^{2q_0q^2}\nonumber \\
&=& a^{2q_0q^2}(\tilde{x}_Q+a)+a^{2q_0}(\tilde{x}_Q+a)+b^{2q_0}+b^{2q_0q^2}+\tilde{z}_Q\nonumber\\
&=& (A^q+A)A^{2q_0q}+A^{2q_0+1}+(A^q+A)^{2q_0}\tilde{x}_Q+O(\tilde{x}_Q^{2q_0}).
\end{eqnarray}

Combining Equation \eqref{tildazPhi2P} with Equations \eqref{wtilda} and \eqref{powerserieswtilda} one gets
\begin{eqnarray} \label{tildawPhi2P}
\tilde{w}_{\Phi^2(Q)}&=& a^{q^3}\tilde{z}_{\Phi^2(Q)}+b^{2q_0q^2}x+w+b^{2q^2}+a^{q^2(2q_0+2)} \nonumber\\
&=& a^{q^3}\tilde{z}_{\Phi^2(Q)}+b^{2q_0q^2}(\tilde{x}_Q+a)+(\tilde{w}_Q+a^q \tilde{z}_Q+b^{2q_0}(\tilde{x}_Q+a)+b^2+a^{2q_0+2})+b^{2q^2}+a^{q^2(2q_0+2)} \nonumber \\
&=& a^{q^3}( (A^q+A)A^{2q_0q}+A^{2q_0+1}+(A^q+A)^{2q_0}\tilde{x}_Q)+b^{2q_0q^2}(\tilde{x}_Q+a) \nonumber\\
&&+b^{2q_0}(\tilde{x}_Q+a)+b^2+a^{2q_0+2}+b^{2q^2}+a^{q^2(2q_0+2)} +O(\tilde{x}_Q^{2q_0})\nonumber \\
&=& P(A)+T(A)^{2q_0}\tilde{x}_Q+O(\tilde{x}_Q^{2q_0}),
\end{eqnarray}
where, $P(A)=A^{q^2+2q_0q+q}+A^{q^2+2q_0q+1}+A^{q^2+2q_0+1}+A^{q+2q_0+1}=A\cdot  T(A)^{2q_0}+A^{q^2+2q_0q+q}.$
Finally,
\begin{eqnarray*} \label{tildazPhiP}
\tilde{z}_{\Phi^(Q)}&=& a^{2q_0q}(\tilde{x}_Q+a)+z+b^{2q_0q}= a^{2q_0q}(\tilde{x}_Q+a)+a^{2q_0}(\tilde{x}_Q+a)+b^{2q_0}+b^{2q_0q}+\tilde{z}_Q\nonumber\\
&=& A^{2q_0+1}+A^{2q_0}\tilde{x}_Q+O(\tilde{x}_Q^{2q_0}) ,
\end{eqnarray*}
and hence
\begin{eqnarray} \label{tildawPhiP}
\tilde{w}_{\Phi(Q)}&=& a^{q^2}\tilde{z}_{\Phi(Q)}+b^{2q_0q}x+w+b^{2q}+a^{q(2q_0+2)} \nonumber\\
&=& a^{q^2}A^{2q_0+1}+a^{q^2}A^{2q_0}\tilde{x}_Q+b^{2q_0q}(\tilde{x}_Q+a)+(b^{2q_0}(\tilde{x}_Q+a)+b^2+a^{2q_0+2})+b^{2q}+a^{q(2q_0+2)}+O(\tilde{x}_Q^{2q_0}) \nonumber \\
&=&A^{q+2q_0+1}+A^{q+2q_0}\tilde{x}_Q+O(\tilde{x}_Q^{2q_0}).
\end{eqnarray}

Equations \eqref{powerseriesytildaPhi2P}, \eqref{tildawPhi2P} and \eqref{tildawPhiP} show that finding $\tilde \alpha$, $\tilde \beta$ and $\tilde \gamma$ such that $v_Q(f_{\tilde \alpha, \tilde \beta, \tilde \gamma})=q_0$ is equivalent to find a solution to the following system of linear equations (corresponding to the coefficients of $\tilde{x}_Q^0$, $\tilde{x}_Q$ and $\tilde{x}_Q^{q_0}$ in the power series expansion of $f_{\alpha,\beta,\gamma}$ respectively),

\begin{equation} \label{system}
\begin{cases}
\alpha A^{q+2q_0+1}+\beta T(A)+\gamma [A\cdot T(A)^{2q_0}+A^{q^2+2q_0q+q}]=0, \\
\alpha A^{q+2q_0}+\beta (A^q+A)^{q_0}+\gamma T(A)^{2q_0}=0, \\
\beta A \ne 0.
\end{cases}
\end{equation}
Since $A \ne 0$ and $(A^q+A)^{q_0}=(T(A)+A^{q_0q+q})/A$, System \eqref{system} can be rewritten as
\begin{equation} \label{system2}
\begin{cases}
\beta=A^{q^2+q_0q} \gamma,\\
\alpha=\frac{\gamma(A^{q^2+q_0q} \cdot T(A)+P(A))}{A^{q+2q_0+1}},\\
\beta\ne 0.
\end{cases}
\end{equation}
To conclude the proof it is now sufficient to define $\tilde{f}_1=f_{\frac{A^{q^2+q_0q}\cdot T(A)+P(A)}{A^{q+2q_0+1}},A^{q^2+q_0q},1}$.
\end{proof}

\begin{lemma} \label{thirdpair}
There exists a function $f_2$ on $\mathcal{S}$ that is regular outside $Q_\infty$ and such that
$$v_Q(f_2)=2q_0q+2q_0+1, \quad \textrm{and} \quad v_{Q_\infty}(f_2) \geq -2q_0(q+2q_0+1).$$
In particular $v_P(f_2)=2q_0q+2q_0+1$, $v_{P_\infty}(f_2) \geq -2q_0(q^2+1)$ and $f_2 \in L((2g(\SK)-2)P_\infty)$.
\end{lemma}

\begin{proof}
We have the following linear equivalence of divisors on $\mathcal{S}$,
\begin{eqnarray*}
(2q_0q+2q_0+1)Q-2q_0(q+2q_0+1)Q_\infty &\cong& (2q_0q+2q_0+1)Q-2q_0(q+2q_0+1)Q_\infty-(\tilde{w}_Q^{2q_0})_{\mathcal{S}}+(\tilde{w}_{\Phi(Q)}^{2})_{\mathcal{S}}\\
&=& (2q_0+1)Q+ 2q_0 \Phi^2(Q)+2\Phi^3(Q)-2(q+2q_0+1)Q_\infty.
\end{eqnarray*}
Hence if we find a function $\tilde{f}_2$ such that $v_Q(\tilde{f}_2)=2q_0+1$, $v_{\Phi^2(Q)}(\tilde{f}_2) \geq 2q_0$, $v_{\Phi^3(Q)}(\tilde{f}_2) \geq 2$ and $v_{Q_\infty}(\tilde{f}_2) \geq -2(q+2q_0+1)$ then we can simply define
$f_2=\tilde{f}_2 \cdot \frac{\tilde{w}_Q^{2q_0}}{\tilde{w}_{\Phi(Q)}^2},$
to complete the proof.

To construct $\tilde{f}_2$, we first observe that
\begin{eqnarray*}
(\tilde{w}_{\Phi^2(Q)} \cdot \tilde{z}_Q)_{\mathcal{S}}
=2q_0 Q +q \Phi^2(Q)+2q_0 \Phi^3(Q)+(E+\Phi^4(Q)) - [2(q+2q_0+1)-1]Q_\infty,
\end{eqnarray*}
where $E$ is effective and not containing $Q$, $\Phi^2(Q)$ and $\Phi^3(Q)$.
Also, from the proof of Lemma \ref{firstpair} we have a function $\tilde{f}_1=\frac{A^{q^2+q_0q}\cdot T(A)+P(A)}{A^{q+2q_0+1}} \tilde{w}_{\Phi(Q)}+A^{q^2+q_0q} \tilde{y}_{\Phi^2(Q)}+ \tilde{w}_{\Phi^2(Q)}$, such that
$(\tilde{f}_1^2)=2q_0 Q+ 2q_0 \Phi^2(Q)+2\Phi^3(Q)+E_1-2(q+2q_0+1)Q_\infty$,
where $E_1$ is effective and not containing $Q$.

Hence if $h_{\alpha,\beta}$ denotes a linear combination $h_{\alpha,\beta}:=\alpha \tilde{w}_{\Phi^2(Q)} \cdot \tilde{z}_Q+\beta \tilde{f}_1^2$ with $\alpha,\beta \in \overline{\mathbb{F}}_{q}$, then

$$v_{\Phi^2(Q)}(h_{\alpha,\beta}) \geq \min\{v_{\Phi^2(Q)}(\tilde{w}_{\Phi^2(Q)} \cdot \tilde{z}_Q), v_{\Phi^2(Q)}( \tilde{f}_1)\}=\min\{q,2q_0\}=2q_0$$
and similarly,
$$v_{\Phi^3(Q)}(h_{\alpha,\beta}) \geq \min\{v_{\Phi^3(Q)}(\tilde{w}_{\Phi^2(Q)} \cdot \tilde{z}_Q), v_{\Phi^3(Q)}( \tilde{f}_1)\}=\min\{2q_0,2\}=2.$$
Also, from Equations \eqref{ytilda}-\eqref{wtilda}, $h_{\alpha,\beta}$ is regular outside $Q_\infty$ and $v_{Q_\infty}(h_{\alpha,\beta}) \geq -2(q+2q_0+1)$.
Hence using the local description of the functions $\tilde{w}_{\Phi^2(Q)}$, $\tilde{z}_Q$ and $\tilde{f}_1$, we find a linear combination $h_{\tilde \alpha,\tilde \beta}$ such that $v_Q(h_{\tilde \alpha, \tilde \beta}) =2q_0+1$ and doing so, it will be enough to define $\tilde f_2:=h_{\tilde \alpha, \tilde \beta}$ to complete the proof.

Systems \eqref{system} and \eqref{system2} yield, $\tilde{f}_1^2=A^{2(q^2+q_0q+1)}\tilde{x}_Q^{2q_0}+O(\tilde{x}_Q^{2q_0+2}),$
while from Equations \eqref{powerseriesztilda} and \eqref{tildawPhi2P} one has,
\begin{eqnarray*}
\tilde{w}_{\Phi^2(Q)} \cdot \tilde{z}_Q&=&(P(A)+T(A)^{2q_0}\tilde{x}_Q+O(\tilde{x}_Q^{2q_0}))\cdot (A\tilde{x}_Q^{2q_0}+\tilde{x}_Q^{2q_0+1}+O(\tilde{x}_Q^q))\\
&=&A\cdot P(A) \tilde{x}_Q^{2q_0}+[P(A)+A\cdot T(A)^{2q_0}]\tilde{x}_Q^{2q_0+1}+O(\tilde{x}_Q^{2q_0+2}).
\end{eqnarray*}
So it is enough to define
\begin{eqnarray*}
\tilde{f}_2:=h_{1,AP(A)/A^{2(q^2+q_0q+1)}}&=&A\cdot P(A) \tilde{x}_Q^{2q_0}+[P(A)+A\cdot T(A)^{2q_0}]\tilde{x}_Q^{2q_0+1}+A\cdot P(A) \tilde{x}_Q^{2q_0}+O(\tilde{x}_Q^{2q_0+2})\\
&=& [P(A)+A\cdot T(A)^{2q_0}]\tilde{x}_Q^{2q_0+1}+O(\tilde{x}_Q^{2q_0+2}).
\end{eqnarray*}
Recalling that $P(A)+AT(A)^{2q_0}=A^{q^2+2q_0q+q} \ne 0$, we see that $v_Q(\tilde{f}_2)=2q_0+1$ as desired.
\end{proof}
\begin{lemma} \label{lastfamily}
For all $i=0,\ldots,q_0-2$ there exists a function $g_i$ on $\mathcal{S}$ that is regular outside $Q_\infty$ and such that
$$v_Q(g_i)=(2i+1)q_0q+i+1, \quad \textrm{and} \quad v_{Q_\infty}(g_i) \geq -((2i+1)q_0-i)(q+2q_0+1).$$
In particular $v_P(g_i)=(2i+1)q_0q+i+1$, $v_{P_\infty}(g_i) \geq -((2i+1)q_0-i)(q^2+1)$ and $g_i \in L((2g(\SK)-2)P_\infty)$ for all $i=0,\ldots,q_0-2$.
\end{lemma}

\begin{proof}
Let $D_i:=((2i+1)q_0q+i+1)Q-((2i+1)q_0-i)(q+2q_0+1)Q_\infty$. Then,
\begin{eqnarray*}
D_i &\cong& ((2i+1)q_0q+i+1)Q-((2i+1)q_0-i)(q+2q_0+1)Q_\infty -(\tilde{w}_Q^{(2i+1)q_0})_{\mathcal{S}}+(\tilde{w}_{\Phi(Q)}^{2i+1})_{\mathcal{S}}\\
&=& (i+1)Q+ (2i+1)q_0 \Phi^2(Q)+(2i+1)\Phi^3(Q)-(i+1)(q+2q_0+1)Q_\infty.
\end{eqnarray*}
Hence if we find a function $\tilde{g}_i$ such that $v_Q(\tilde{g}_i)=i+1$, $v_{\Phi^2(Q)}(\tilde{g}_i) \geq (2i+1)q_0$, $v_{\Phi^3(Q)}(\tilde{g}_i) \geq 2i+1$ and $v_{Q_\infty}(\tilde{g}_i) \geq ((2i+1)q_0-i)(q+2q_0+1)$ then we simply define $g_i=\tilde{g}_i \cdot \bigg(\frac{\tilde{w}_Q^{q_0}}{\tilde{w}_{\Phi(Q)}}\bigg)^{2i+1}$ to complete the proof.

To this end we first compute some local power series expansions at $\Phi^3(P)$ (equivalently $\Phi^3(Q)$) using the local parameter $\tilde{x}_{\Phi^3(Q)}=x+a^{q^3}=\tilde{x}_Q+a+a^{q^3}=\tilde{x}_Q+A+A^q+A^{q^2}$.
Clearly, we can use the expressions we already obtained in Equations \eqref{powerseriesyb}-\eqref{powerserieswtilda} to describe the local power series expansion of $y+b^{q^3}$, $\tilde{y}_{\Phi^3(Q)}$, $\tilde{z}_{\Phi^3(Q)}$, $\tilde{w}_{\Phi^3(Q)}$ at $\Phi^3(P)$  simply replacing $a$ with $a^{q^3}$ and hence $A$ with $A^{q^3}$.
From Equation \eqref{powerseriesyb} and $a^{q_0}A=b^q+b$, one has
$
y+b^{q^2}= (y+b^{q^3})+(b^{q^3}+b^{q^2})=a^{q_0q^2}A^{q^2}+a^{q_0q^3} \tilde{x}_{\Phi^3(Q)} +A^{q^3} \tilde{x}_{\Phi^3(Q)}^{q_0}+ \tilde{x}_{\Phi^3(Q)}^{q_0+1}+O( \tilde{x}_{\Phi^3(Q)}^{q}).
$
Hence
\begin{eqnarray} \label{powerseriesphi3P1}
\tilde{y}_{\Phi^2(Q)}&=&(y+b^{q^2})+a^{q_0q^2}(x+a^{q^2})=(y+b^{q^2})+a^{q_0q^2}(\tilde{x}_{\Phi^3(Q)}+A^{q^2}) \nonumber \\
&=&A^{q_0q^2} \tilde{x}_{\Phi^3(Q)} +A^{q^3} \tilde{x}_{\Phi^3(Q)}^{q_0}+ \tilde{x}_{\Phi^3(Q)}^{q_0+1}+O( \tilde{x}_{\Phi^3(Q)}^{q}).
\end{eqnarray}
Similarly using Equation \eqref{powerseriesztilda} and $z=\tilde{z}_{\Phi^3(Q)}+a^{2q_0q^3}x+b^{2q_0q^3}$ one gets,
\begin{eqnarray} \label{powerseriesphi3P2}
\tilde{z}_{\Phi^2(Q)}&=& a^{2q_0q^2}( \tilde{x}_{\Phi^3(Q)}+a^{q^3})+z+b^{2q_0q^2} \nonumber \\
&=& a^{2q_0q^2}( \tilde{x}_{\Phi^3(Q)}+a^{q^3})+\tilde{z}_{\Phi^3(Q)}+a^{2q_0q^3}x+b^{2q_0q^3}+b^{2q_0q^2} \nonumber \\
&=& A^{2q_0q^2} \tilde{x}_{\Phi^3(Q)}+\tilde{z}_{\Phi^3(Q)}=A^{2q_0q^2} \tilde{x}_{\Phi^3(Q)}+O(\tilde{x}_{\Phi^3(Q)}^{2q_0}).
\end{eqnarray}
We define $H_1:=\tilde{z}_{\Phi^2(Q)}+A^{q_0q^2}\tilde{y}_{\Phi^2(Q)}$. Then using Equations \eqref{powerseriesphi3P1} and \eqref{powerseriesphi3P2},
\begin{eqnarray*} \label{H1}
H_1&=&A^{q_0q^2}( A^{q_0q^2}\tilde{x}_{\Phi^3(Q)} +A^{q^3} \tilde{x}_{\Phi^3(Q)}^{q_0}+ \tilde{x}_{\Phi^3(Q)}^{q_0+1}) + A^{2q_0q^2} \tilde{x}_{\Phi^3(Q)}+O( \tilde{x}_{\Phi^3(Q)}^{2q_0})\nonumber \\
&=&A^{q^3+q_0q^2} \tilde{x}_{\Phi^3(Q)}^{q_0}+ A^{q_0q^2}\tilde{x}_{\Phi^3(Q)}^{q_0+1}+O( \tilde{x}_{\Phi^3(Q)}^{2q_0}).
\end{eqnarray*}
Since $A \ne 0$ this shows that $v_{\Phi^3(Q)}(H_1)=q_0$. Also note that
$$v_{\Phi^2(Q)}(H_1) = \min\{v_{\Phi^2(Q)}(\tilde{z}_{\Phi^2(Q)}), v_{\Phi^2(Q)}( \tilde{y}_{\Phi^2(Q)})\}=\min\{2q_0,q_0\}=q_0.$$
To compute the valuation of $H_1$ at $P$ we see that from Equations \eqref{powerseriesytildaPhi2P} and \eqref{tildazPhi2P},
\begin{eqnarray*}
H_1=\tilde{z}_{\Phi^2(Q)}+A^{q_0q^2}\tilde{y}_{\Phi^2(Q)}
=P_1(A)+[(A^q+A)^{2}+A^{q^2}(A^q+A)]^{q_0}\tilde{x}_Q+O(\tilde{x}_Q^{q_0}),
\end{eqnarray*}
where $P_1(A)=(A^q+A)A^{2q_0q}+A^{2q_0+1}+A^{q_0q^2}T(A)$.
Note that the coefficient of $\tilde{x}_Q$ in $H_1$ is not zero as both $A^q+A= 0$ and $A^{q^2}+A^q+A=0$ imply $P \in \SK(\mathbb{F}_{q^4})$.

If $P_1(A)=0$ then $v_Q(H_1)=1$ and we define $N_1:=H_1$. Otherwise recalling that from Equation \eqref{tildawPhi2P} we have $\tilde{w}_{\Phi^2(Q)}=P(A)+T(A)^{2q_0}\tilde{x}_Q+O(\tilde{x}_Q^{2q_0}),$
we define $N_1:=P_1(A)\tilde{w}_{\Phi^2(Q)}+P(A)H_1.$
Doing so,
\begin{eqnarray*}
N_1&=&P_1(A)(P(A)+T(A)^{2q_0}\tilde{x}_Q+O(\tilde{x}_Q^{2q_0}))+P(A)(P_1(A)+[(A^q+A)^{2}+A^{q^2}(A^q+A)]^{q_0}\tilde{x}_Q+O(\tilde{x}_Q^{q_0}))\\
&=&P_2(A)\tilde{x}_Q+O(\tilde{x}_Q^{q_0}).
\end{eqnarray*}
Using $P(A)=AT(A)^{2q_0}+A^{q^2+2q_0q+q}$, $(A^q+A)^{q_0}=(T(A)+A^{q_0q+q})/A$ and $P \not\in \SK(\mathbb{F}_{q^4})$, one obtains
$
P_2(A)=P_1(A)T(A)^{2q_0}+P(A)[(A^q+A)^{2}+A^{q^2}(A^q+A)]^{q_0}
=A^{q_0q+q+q_0}P(A)^{q_0} \ne 0.
$

Since $N_1$ is a linear combination of $\tilde{w}_{\Phi^2(Q)}$ and $H_1$, both the valuations of $N_1$ at $\Phi^2(Q)$ and $\Phi^3(Q)$ are at least $q_0$.
This shows that, in both cases $P_1(A)=0$ and $P_1(A) \ne 0$, one has
\begin{equation} \label{n1}
(N_1)_{\mathcal{S}}=Q+q_0\Phi^2(Q)+q_0\Phi^3(Q)+E_{N_1}-(q+2q_0+1)Q_\infty,
\end{equation}
where $E_{N_1}$ is effective and with support not containing $Q$.
Finally, we show that there exists a function $N_0$ on $\mathcal{S}$ such that
\begin{equation} \label{n0}
(N_0)_{\mathcal{S}}=Q+2q_0\Phi^2(Q)+\Phi^3(Q)+E_{N_0}-(q+2q_0+1)Q_\infty,
\end{equation}
where $E_{N_0}$ is effective and with support not containing $Q$. Indeed since we already obtained in Equations \eqref{tildazPhi2P} and \eqref{tildawPhi2P} that
$\tilde{z}_{\Phi^2(Q)}=(A^q+A)A^{2q_0q}+A^{2q_0+1}+(A^q+A)^{2q_0}\tilde{x}_Q+O(\tilde{x}_Q^{2q_0}),$
and
$\tilde{w}_{\Phi^2(Q)}=P(A)+T(A)^{2q_0}\tilde{x}_Q+O(\tilde{x}_Q^{2q_0}),$
we can define $N_0:=P(A)\tilde{z}_{\Phi^2(Q)}+[(A^q+A)A^{2q_0q}+A^{2q_0+1}]\tilde{w}_{\Phi^2(Q)}$. Doing so,
$$
N_0=
(P(A)(A^q+A)^{2q_0}+[(A^q+A)A^{2q_0q}+A^{2q_0+1}]T(A)^{2q_0})\tilde{x}_Q+O(\tilde{x}_Q^{2q_0}),\\
$$
where by direct checking $P(A)(A^q+A)^{2q_0}+[(A^q+A)A^{2q_0q}+A^{2q_0+1}]T(A)^{2q_0}=A^{2q_0q+2q+2q_0} \ne 0$,
and hence $v_Q(N_0)=1$.
Also, $v_{\Phi^2(Q)}(N_0) = \min\{v_{\Phi^2(Q)}(\tilde{z}_{\Phi^2(Q)}), v_{\Phi^2(Q)}( \tilde{w}_{\Phi^2(Q)})\}=\min\{2q_0,q\}=2q_0,$ and $v_{\Phi^3(Q)}(N_0) = \min\{v_{\Phi^3(Q)}(\tilde{z}_{\Phi^2(Q)}), v_{\Phi^3(Q)}( \tilde{w}_{\Phi^2(Q)})\}=\min\{1,2q_0\}=1,$ as desired.
Using Equations \eqref{n1} and \eqref{n0} it is now enough to define $\tilde{f}_i=N_0^i \cdot N_1$. Indeed one has,
$$
(\tilde{f}_i)_{\mathcal{S}}
=(i+1)Q+(2i+1)q_0\Phi^2(Q)+(q_0+i)\Phi^3(Q)+(iE_{N_0}+E_{N_1})-(i+1)(q+2q_0+1)Q_\infty,
$$
where $q_0+i>2i+1$ as $i \leq q_0-2$.
\end{proof}

To construct gaps using Corollary \ref{basisdiff} we will look at functions of the form
\begin{equation} \label{eq:functions}
  \tilde{x}_Q^{a_1} \cdot \tilde{y}_Q^{a_2} \cdot \tilde{z}_Q^{a_3} \cdot \tilde{w}_Q^{a_4} \cdot \prod_{n=1}^{2q_0-2} h_n^{b_n} \cdot f_1^c \cdot f_2^d \cdot \prod_{n=0}^{q_0-2} g_n^{e_n} \cdot \pi_P^f
\end{equation}
for suitable choices of exponents $a_1, a_2, a_3, a_4, b_1, \dots, b_{2q_0-2}, c, d, e_0, \dots, e_{q_0-2}, f$.

The six families of natural numbers $F_1,F_2,F_3,F_4,F_5,F_6$ defined in Theorem \ref{all} correspond to sets of valuations of functions as in Equation \eqref{eq:functions}. Indeed our aim is to show that $G(P) = F := \cup_{i=1}^6 F_i$.
To this end we proceed with the following two steps: first, we prove that $F$ contains exactly $g(\SK)$ elements. Then we prove that the functions as in Equation \eqref{eq:functions} whose valuations are contained in $F$, are in $L((2g(\SK)-2)P_\infty)$.
\begin{proposition} \label{distinctgaps}
 The set $F$ consists of $g(\SK)$ distinct natural numbers.
\end{proposition}

\begin{proof}
If $s=1,2$ the claim can be checked directly using a computer. Table \ref{tab:table1} collects the cardinalities of the families $F_i's$ in these two cases.
\begin{table}[h!]
  \begin{center}
    \caption{The cases $s=1$ and $s=2$.}
    \label{tab:table1}
    \begin{tabular}{l|c|c|c|c|c|c|c|r}
      $s$ & $|F_1|$ & $|F_2|$ & $|F_3|$ & $|F_4|$ & $|F_5|$ & $|F_6|$ & $|F|$ & $g(\SK)$ \\
     \hline
      $1$ & $146 $ & $31$ & $8$ & $0$ & $9$ & $2$ & $196$ & $196$ \\
      $2$ & $12584$ & $2393$ & $192$ & $96$ & $87$ & $24$ & $15376$ & $15376$ \\
      \hline
    \end{tabular}
  \end{center}
\end{table}

So in the following, we assume $s>2$.
  We first show that for all $i=1,\dots,6$ different choices of coefficients within the same family $F_i$ give rise to distinct natural numbers.
  \begin{itemize}
  \item For $i=1$, assume by contradiction that there exists element $v \in F_1$ having two different expressions, namely
    \begin{equation} \label{eq:distinctF1}
      a_1 + a_2 q_0 + 2 a_3 q_0 + a_4q + fq^2 + 1 = a_1^\prime + a_2^\prime q_0 + 2 a_3^\prime q_0 + a_4^\prime q + f^\prime q^2 + 1.
    \end{equation}
    Considering Equation \eqref{eq:distinctF1} modulo $q_0$ we obtain $a_1 \equiv  a_1^\prime \mod q_0$ and therefore $a_1 = a_1^\prime$, since $0 \leq a_1, a_1^\prime \leq q_0-1$. So, Equation \eqref{eq:distinctF1} can be simplified as $ a_2 + 2 a_3 + 2a_4 q_0 + 2f qq_0 = a_2^\prime + 2 a_3^\prime + 2a_4^\prime q_0 + 2f^\prime qq_0. $
    Repeating the same procedure, firstly considering the equality modulo $2$, then modulo $q_0$ again and finally modulo $q$, we obtain $a_2 = a_2^\prime, a_3 = a_3^\prime, a_4 = a_4^\prime$ respectively. As a consequence, $f = f^\prime$. This yields a contradiction.
  \item As before, for $i=2$ let
    $$ a_1 + a_2 q_0 + 2 a_3 q_0 + a_4q + fq^2 + (n+1)qq_0 + 1 = a_1^\prime + a_2^\prime q_0 + 2 a_3^\prime q_0 + a_4^\prime q + f^\prime q^2 + (n^\prime+1)qq_0 + 1. $$
    With a similar argument used for $i=1$, we can reduce the above equality modulo $q_0$, then modulo 2, then modulo $q_0$ again and finally modulo $q$ to obtain $a_1 = a_1^\prime, a_2 = a_2^\prime, a_3 = a_3^\prime$ and $a_4 + nq_0 = a_4^\prime + n^\prime q_0$ respectively. As a consequence $f = f^\prime$. The conditions $a_1 + a_2 + a_3 + a_4 + f = q-q_0-2-nq_0+n$ and $a_1^\prime + a_2^\prime + a_3^\prime + a_4^\prime + f^\prime = q-q_0-2-n^\prime q_0+n^\prime$ now imply that $n = n^\prime$ and $a_4 = a_4^\prime$.
\item The argument used for $i=3,\ldots,6$ is exactly the same as for $i=1,2$ and hence it is omitted.
  \end{itemize}
  Next, one needs to prove that $F_i \cap F_j = \emptyset$, for all $i,j=1,\dots,6, i \neq j$. To this end, one can use exactly the same method used in the first part of this proof. For this reason just the first two cases are proven here in full details. Let $\sigma = a_1 + a_2 + a_3 + a_4 + f$.
  \begin{itemize}
  \item $F_1 \cap F_2 = \emptyset$: suppose by contradiction that there exists $v \in F_1\cap F_2$. Then:
    \begin{align*}
      v & = a_1+a_2 q_0+2a_3 q_0 + a_4 q + (\sigma-a_1-a_2-a_3-a_4) q^2+1 \\
        & = a'_1+a'_2 q_0+2a'_3 q_0 + a'_4 q + (n+1)q_0 q + (q-q_0-2-nq_0+n-a'_1-a'_2-a'_3-a'_4)q^2+1.
    \end{align*}
    Considering the equality above modulo $q_0$, then modulo $2$ and finally modulo $q_0$ we get $a_1=a_1^\prime, a_2=a'_2,a_3=a'_3$. This yields
    $$ a_4+ (\sigma-a_4) q = a'_4 + (n+1)q_0 + (q-q_0-2-nq_0+n-a'_4)q. $$
    Since $a_4 \leq q-2$ and $a'_4+(n+1)q_0 \leq q-q_0-1-nq_0+nq_0+q_0=q-1$, the equality above modulo $q$ gives $a_4=a'_4+(n+1)q_0$ and $\sigma = q-2+n$, a contradiction to $\sigma \leq q-2$ as $n \geq 1$.
  \item $F_1 \cap F_3 = \emptyset$: suppose by contradiction that there exists $v \in F_1\cap F_3$. Then we can write
    \begin{align*}
      v & = a_1+a_2 q_0+2a_3 q_0 + a_4 q + (\sigma -a_1-a_2-a_3-a_4)q^2+1 \\
        & = a'_1+a'_2 q_0+2a'_3 q_0 + a'_4 q +(2n+1)q_0q + n +1 +(q-q_0-2-2nq_0+n-a'_1-a'_2-a'_3-a'_4)q^2+1.
    \end{align*}
    Arguing as in the previous case one gets $a_1=a'_1+n+1, a_2=a'_2$ and $a_3=a'_3$, so that
    $$ a_4 + (\sigma -n-a_4)q= a'_4+(2n+1)q_0+(q-q_0-1-2nq_0+n-a'_4)q.$$
    Note that $a'_4+(2n+1)q_0 \leq q_0-1+(2q_0-3)q_0<q$ so that considering the equality modulo $q$ one gets $a_4=a'_4+(2n+1)q_0$. Hence $\sigma =q-1+2n \geq q-1$, a contradiction.
\end{itemize}

Finally, we need to prove that $F$ consists of $g(\SK)$ distinct elements. From the previous step it is enough to check that there are exactly $g(\SK)$ possible distinct choices of such coefficients.

For a given $\sigma \in \mathbb{N}$, we denote the number of combinations of five natural numbers $a_1, a_2, a_3, a_4, f$ such that $a_1 + a_2 + a_3 + a_4 + f = \sigma$ with $\mathcal{B}(\sigma) = \binom{\sigma+4}{4}$.
Observe that from \cite[Equation 5.10]{CM} for any $n \in \mathbb{N}$ it holds:
$$
  \sum_{\sigma=0}^{n} \mathcal{B}(\sigma) =\sum_{\sigma'=4}^{n+4} {{\sigma'}\choose{4}}={{n+5}\choose{5}}
   = \frac{n(n+1)}{120} (n^3 + 14n^2 + 71n + 154) + 1.
$$
It is also useful to denote the number of combinations of four natural numbers $a_2, a_3, a_4, f$ such that $a_2 + a_3 + a_4 + f = \sigma$ with $\mathcal{B}^\prime(\sigma) = \binom{\sigma+3}{3}$ and, similarly, the number of combinations of three natural number $a_3, a_4, f$ such that $a_3 + a_4 + f = \sigma$ with $\mathcal{B}^{\prime\prime}(\sigma) = \binom{\sigma+2}{2}$.
\begin{itemize}
\item Let us count the number of elements in $F_1$. By the inclusion-exclusion principle, this is given by:
  \begin{align*}
    & \sum_{\sigma=0}^{q-2} \mathcal{B}(\sigma) - \left( \sum_{\sigma_1=0}^{q-2-q_0} \mathcal{B}(\sigma_1) + \sum_{\sigma_2=0}^{q-2-2} \mathcal{B}(\sigma_2) + \sum_{\sigma_3=0}^{q-2-q_0} \mathcal{B}(\sigma_3) \right) \\
    & + \left( \sum_{\sigma_{1,2}=0}^{q-2-q_0-2} \mathcal{B}(\sigma_{1,2}) + \sum_{\sigma_{2,3}=0}^{q-2-2-q_0} \mathcal{B}(\sigma_{2,3}) + \sum_{\sigma_{1,3}=0}^{q-2-q_0-q_0} \mathcal{B}(\sigma_{1,3}) \right) - \sum_{\sigma_{1,2,3}=0}^{q-2-q_0-2-q_0} \mathcal{B}(\sigma_{1,2,3}) \\
    & = \frac{q^3}{2} - q^2q_0 + \frac{7}{24}q^2 - \frac{1}{12}q.
  \end{align*}

\item Let us now count the number of elements in $F_2$. In this case $\sigma=q-q_0-2-nq_0+n$ in $F_2$. Observe that:
$\sigma-q_0, \sigma-2$ are always non-negative, while $\sigma-(q-q_0-nq_0)$ is non-negative only for $n > 1$;
$\sigma-2q_0, \sigma-2-q_0$ are non-negative as $s > 1$; $\sigma-(q-q_0-nq_0)-2$ and $\sigma-(q-q_0-nq_0) - q_0$ are non-negative only for $n > 3$ and $n > q_0+1$, respectively. Both conditions can hold true as $s > 1$;
$\sigma-2q_0-2$ is non-negative since $s > 1$; $\sigma-(q-q_0-nq_0)-2-q_0$ is non-negative only for $n > q_0 + 3$, which can hold true as $s > 2$; $\sigma-(q-q_0-nq_0)-2q_0$ is always negative; and $\sigma-(q-q_0-nq_0)-2-2q_0$ is always negative.
Hence, the number of elements in $F_2$ is given by:
  \begin{align*}
    & \sum_{n=1}^{2q_0-2} \mathcal{B}(\sigma) - \left( \sum_{n=1}^{2q_0-2} \mathcal{B}(\sigma-2) + \sum_{n=1}^{2q_0-2} 2\mathcal{B}(\sigma-q_0) + \sum_{n=2}^{2q_0-2} \mathcal{B}(\sigma-(q-q_0-nq_0)) \right) + \Bigg( \sum_{n=1}^{2q_0-2} \mathcal{B}(\sigma-2q_0)\\
    & + \sum_{n=1}^{2q_0-2} 2\mathcal{B}(\sigma-2-q_0) + \sum_{n=4}^{2q_0-2} \mathcal{B}(\sigma-(q-q_0-nq_0)-2)+ \sum_{n=q_0+2}^{2q_0-2} 2\mathcal{B}(\sigma-(q-q_0-nq_0)-q_0) \Bigg) \\
    &  - \left( \sum_{n=1}^{2q_0-2} \mathcal{B}(\sigma-2-2q_0) + \sum_{n=q_0+4}^{2q_0-2} 2\mathcal{B}(\sigma-(q-q_0-nq_0)-2-q_0) \right) \\
    & = q^2q_0 - \frac{43}{24} q^2 + qq_0 + \frac{1}{12} q + 1.
  \end{align*}
\item For counting the elements of $F_3$, observe that $\sigma-q_0, \sigma-2, \sigma-(q_0-1-n), \sigma-2-q_0, \sigma-2q_0, \sigma-(q_0-1-n)-2, \sigma-(q_0-1-n)-q_0$ are always non-negative, and also $\sigma-2-2q_0, \sigma-2-q_0-(q_0-1-n), \sigma-2q_0-(q_0-1-n), \sigma-2-2q_0-(q_0-1-n)$ are non-negative as $s > 1$.

Taking into consideration the aforementioned observations we get that the cardinality of $F_3$ is equal to
  \begin{align*}
    & \sum_{n=0}^{q_0-2} \mathcal{B}(\sigma) - \sum_{n=0}^{q_0-2} \left( \mathcal{B}(\sigma-2) + 2\mathcal{B}(\sigma-q_0) + \mathcal{B}(\sigma-(q_0-1-n)) \right) \\
    & + \sum_{n=0}^{q_0-2} \left( \mathcal{B}(\sigma-2q_0) + 2\mathcal{B}(\sigma-2-q_0) + \mathcal{B}(\sigma-(q_0-1-n)-2) + 2\mathcal{B}(\sigma-(q_0-1-n)-q_0) \right) \\
    & - \sum_{n=0}^{q_0-2} \left( \mathcal{B}(\sigma-2-2q_0) + 2\mathcal{B}(\sigma-2-q_0-(q_0-1-n)) + \mathcal{B}(\sigma-2q_0-(q_0-1-n)) \right) \\
    & + \sum_{n=0}^{q_0-2} \mathcal{B}(\sigma-2-2q_0-(q_0-1-n))  = \frac{1}{4} q^2 - \frac{1}{2} qq_0.
  \end{align*}
\item The number of elements of $F_4$ is given by:
  \begin{align*}
    & \sum_{n=0}^{q_0-3} \mathcal{B}(\sigma) - \sum_{n=0}^{q_0-3} \left( \mathcal{B}(\sigma-2) + 2\mathcal{B}(\sigma-q_0) + \mathcal{B}(\sigma-(q_0-2-n)) \right) \\
    & + \sum_{n=0}^{q_0-3} \left( \mathcal{B}(\sigma-2q_0) + 2\mathcal{B}(\sigma-2-q_0) + \mathcal{B}(\sigma-(q_0-2-n)-2) + 2\mathcal{B}(\sigma-(q_0-2-n)-q_0) \right) \\
    & - \sum_{n=0}^{q_0-3} \left( \mathcal{B}(\sigma-2-2q_0) + 2\mathcal{B}(\sigma-2-q_0-(q_0-2-n)) + \mathcal{B}(\sigma-2q_0-(q_0-2-n)) \right) \\
    & + \sum_{n=0}^{q_0-3} \mathcal{B}(\sigma-2-2q_0-(q_0-2-n)) = \frac{1}{4} q^2 - \frac{3}{2} qq_0 + q.
  \end{align*}
\item For studying the number of elements of $F_5$, let us consider the cases $c=0$ and $c=1$ separately, starting from the case $c=0$.

The number of elements in $F_5$ for $c=0$ can be computed as follows
  \begin{align*}
    & \sum_{d=1}^{q_0-1} \mathcal{B}^\prime(\sigma) - \sum_{d=1}^{q_0-1} \left( \mathcal{B}^\prime(\sigma-2) + \mathcal{B}^\prime(\sigma-q_0) + \mathcal{B}^\prime(\sigma-(q_0-d)) \right) \\
    & + \sum_{d=1}^{q_0-1} \left( \mathcal{B}^\prime(\sigma-2-q_0) + \mathcal{B}^\prime(\sigma-2-(q_0-d)) + \mathcal{B}^\prime(\sigma-q_0-(q_0-d)) \right) \\
    & - \sum_{d=1}^{q_0-1} \mathcal{B}^\prime(\sigma-2-q_0-(q_0-d)) = \frac{1}{2} qq_0 - \frac{1}{2} q.
  \end{align*}
  Let us consider now the case $c=1$. Observe that $\sigma-q_0$ and $\sigma-q_0-(q_0-d)$ are negative only for $d = q_0-1$, while the quantity $\sigma-(q_0-d)$ is always non-negative as $s>2$.
Hence the number of elements in $F_5$ with $c=1$ is given by
  \begin{align*}
    & \sum_{d=0}^{q_0-1} \mathcal{B}^{\prime\prime}(\sigma) - \left( \sum_{d=0}^{q_0-1} \mathcal{B}^{\prime\prime}(\sigma-(q_0-d)) + \sum_{d=0}^{q_0-2} \mathcal{B}^{\prime\prime}(\sigma-q_0) \right) + \sum_{d=0}^{q_0-2} \mathcal{B}^{\prime\prime}(\sigma-q_0-(q_0-d)) = \frac{1}{4} qq_0 + \frac{1}{4} q - 1.
  \end{align*}
\item For counting the number of elements of $F_6$, observe that $\sigma - q_0 - (n+1)$ is non-negative as $s>1$.
So, $|F_6|$ coincides with
  \begin{align*}
    & \sum_{n=0}^{q_0-2} \mathcal{B}^{\prime\prime}(\sigma) - \sum_{n=0}^{q_0-2} \left( \mathcal{B}^{\prime\prime}(\sigma-q_0) + \mathcal{B}^{\prime\prime}(\sigma-(n+1)) \right) + \sum_{n=0}^{q_0-2} \mathcal{B}^{\prime\prime}(\sigma-q_0-(n+1)) = \frac{1}{4} qq_0 - \frac{1}{4} q.
  \end{align*}
\end{itemize}
The conclusion follows after adding up the quantities obtained for each of the six families separately. For all $s>2$ the computation yields $|F| = \frac{1}{2}q(q-1)^2=g(\SK)$.
\end{proof}

A consequence of Proposition \ref{distinctgaps} is that we can assign to each $v \in F$ a unique function $\mathfrak{f}_v$ on $\tilde{\mathcal{S}}$ of the form as in Equation \eqref{eq:functions} that is regular outside $P_\infty$ and has valuation $v-1$ at $P$ in the following way: if $v$ is expressed as
$$ v = a_1 + a_2 q_0 + 2a_3 q_0 + a_4 q + qq_0 \sum_{n=1}^{2q_0-2} b_n(n+1) + cq_0(q+1) + d(2qq_0+2q_0+1) + \sum_{n=0}^{q_0-2} e_n((2n+1)q_0q+n+1) + f q^2 + 1, $$
then we define
$$ \mathfrak{f}_v :=  \tilde{x}_Q^{a_1} \cdot \tilde{y}_Q^{a_2} \cdot \tilde{z}_Q^{a_3} \cdot \tilde{w}_Q^{a_4} \cdot \prod_{n=1}^{2q_0-2} h_n^{b_n} \cdot f_1^c \cdot f_2^d \cdot \prod_{n=0}^{q_0-2} g_n^{e_n} \cdot \pi_P^f. $$
Proposition \ref{distinctgaps} guarantees that there are exactly $g(\SK)$ such functions having pairwise distinct valuations at $P$. The following proposition, together with Proposition \ref{distinctgaps} and Corollary \ref{basisdiff}, completes the proof of Theorem \ref{all}.

\begin{proposition}
$\mathfrak{f}_v$ belongs to $L((2g(\SK)-2)P_\infty)$ for all $v \in F$.
\end{proposition}

\begin{proof}
  It has been already shown that $\mathfrak{f}_v$ is regular outside $P_\infty$. It remains to show that $-v_{P_\infty}(\mathfrak{f}_v) \leq 2g(\SK)-2$. As before, let $\sigma = a_1 + a_2 + a_3 + a_4 + f$.
  \begin{itemize}
  \item If $v \in F_1$, then $-v_{P_\infty}(\mathfrak{f}_v)$ is equal to
$
       \sigma(q^2+1) - a_1(2qq_0-q+1) - a_2(qq_0-q_0+1) - a_3(q-2q_0+1) \leq (q-2)(q^2+1) = 2g(\SK)-2.
$
  \item If $v \in F_2$, then $-v_{P_\infty}(\mathfrak{f}_v)$ is equal to
$
        (q-2)(q^2+1) - a_1(2qq_0-q+1) - a_2(qq_0-q_0+1) - a_3(q-2q_0+1) \leq (q-2)(q^2+1) = 2g(\SK)-2.
$
  \item If $v \in F_3$, then $-v_{P_\infty}(\mathfrak{f}_v)$ is at most
$
      (q-2)(q^2+1) - a_1(2qq_0-q+1) - a_2(qq_0-q_0+1) - a_3(q-2q_0+1) \leq (q-2)(q^2+1) = 2g(\SK)-2.
$
  \item If $v \in F_4$, then $-v_{P_\infty}(\mathfrak{f}_v)$ is at most
$
       (q-2)(q^2+1) - a_1(2qq_0-q+1) - a_2(qq_0-q_0+1) - a_3(q-2q_0+1) \leq (q-2)(q^2+1) = 2g(\SK)-2.
$
  \item If $v \in F_5$, then $-v_{P_\infty}(\mathfrak{f}_v)$ is at most
$
       (q-2)(q^2+1) - a_2(qq_0-q_0+1) - a_3(q-2q_0+1) \leq (q-2)(q^2+1) = 2g(\SK)-2.
$
  \item If $v \in F_6$, then $-v_{P_\infty}(\mathfrak{f}_v)$ is at most
$
       (q-2)(q^2+1) - a_3(q-2q_0+1) \leq (q-2)(q^2+1) = 2g(\SK)-2.
$
  \end{itemize}
\end{proof}

It is possible, though a bit technical, to determine the Ap\'ery set and a set of generators of $H(P)$ for $P \not\in \SK(\mathbb{F}_{q^4}).$ It turns out that for $s=1$, one needs $19$ generators, while for $s \ge 2$, $H(P)$ has $3q-2q_0$ generators. More details will appear in the upcoming PhD thesis of the second author.

\section*{Acknowledgements}

The first author would like to acknowledge the support from The Danish Council for Independent Research (DFF-FNU) for the project \emph{Correcting on a Curve}, Grant No.~8021-00030B.

\vspace{1ex}
\noindent
Peter Beelen, Leonardo Landi and Maria Montanucci

\vspace{.5ex}
\noindent
Technical University of Denmark,\\
Department of Applied Mathematics and Computer Science,\\
Matematiktorvet 303B,\\
2800 Kgs. Lyngby,\\
Denmark,\\
pabe@dtu.dk\\
lelan@dtu.dk\\
marimo@dtu.dk\\

\end{document}